\documentclass [10pt]{amsart}
\usepackage{amsmath,amsfonts,amsthm,amssymb,pxfonts}
\usepackage{Tabbing}
\usepackage{lastpage}
\usepackage{mathtools}
\usepackage{esint}
\usepackage{mathrsfs}

\newtheorem{theorem}{Theorem}

\newtheorem{lemma}[theorem]{Lemma}

\newtheorem{corollary}[theorem]{Corollary}
\newtheorem{proposition}[theorem]{Proposition}

\begin{document}

\title[Convergence of Strong Arithmetic Means of Fourier Series]{On almost everywhere convergence of strong arithmetic means of Fourier series}

\author{Bobby Wilson}
\address{Department of Mathematics, The University of Chicago, 5734 South University Avenue, Chicago, IL 60615, U.S.A.}
\thanks{This is in partial fulfillment of the author's requirements for the Doctor of Philosophy degree in Mathematics at the University of Chicago}

\subjclass[2010]{42A20, 42A24}

\keywords{Fourier series on $L^1$, $\ell^2$ averages, Calder\'{o}n-Zygmund decomposition, partial sums, weak $L^1$, strong arithmetic means}

\begin{abstract}
This article establishes a real-variable argument for Zygmund's theorem on almost everywhere convergence of strong arithmetic means of partial sums of Fourier series on $\mathbb{T}$, up to passing to a subsequence. Our approach extends to, among other cases, functions that are defined on $\mathbb{T}^d$, which allows us to establish an analogue of Zygmund's theorem in higher dimensions.
\end{abstract}

\maketitle

\section{Introduction}
Let $\mathbb{T}= \mathbb{R}/\mathbb{Z}$ and denote by $S_nf(\theta)$ the partial sum of the Fourier series of $f$ with respect to the trigonometric system.  
Thus, with $e(\theta):=e^{2\pi i \theta}$  we have for any $f\in L^1(\mathbb{T})$ and nonnegative integer~$n$, 
\[
(S_nf)(\theta)= \sum_{j= -n}^n \widehat{f}\,(j)\, e(j\theta)
\]
The most classical questions in Fourier analysis concern the convergence of this sequence of partial sums. "Convergence" here of course admits many different interpretations. 
The archetypical notion of convergence, i.e., the pointwise sense, was eventually replaced by more flexible and robust ones, such as almost everywhere or $L^p$ convergence. 
Many  investigations into fine convergence properties of Fourier series ensued, such as by the Russian school, Zygmund's school, as well as by Hardy and Littlewood. One result of these efforts was the formulation of the {\em Lusin conjecture} in 1915:
 for every $f\in L^2(\mathbb{T})$ the partial sums $S_n f$ converge almost everywhere.  
In contrast to this assertion,  Kolmogoroff~\cite{kolm23} famously constructed an $L^1$ function in 1923 for which $S_nf$ diverges almost everywhere. 
Cast in modern terminology, he showed that there exists  $f\in L^1(\mathbb{T})$ such that  the {\em Carleson maximal function}  
\begin{align*}
\mathcal{C}f := \sup_{n\ge0} |S_nf|
\end{align*}
satisfies $$|\{\mathcal{C}f<\infty\}|=0, $$ see for example Chapter~6 of~\cite{schlag13}. 
In a major breakthrough,  Carleson~\cite{carl66} proved Lusin's conjecture  in 1966.  This was extended to $f \in L^p(\mathbb{T})$ for any $p>1$ by~Hunt \cite{hun68}. 
For more historical background, as well as for a modern presentation of these results, see  the book~\cite{2schlag13}. 

However, research into Fourier series did not limit itself to the basic investigation of convergence of the partial sum sequence. 
Let us give three examples of finer questions:
\begin{itemize}
\item
How fast may  partial Fourier sums grow for a given $f \in L^1(\mathbb{T})$? 
\item Given $f\in L^1$, what can we say about the density of possible subsequences $\{n_k\}$ in $\mathbb{N}$
for which $\{S_{n_k}f\}$ does converge?
\item 
  Does there exists a sequence $\{M_j\}$ of positive integers such that for the Fourier series of any $f \in L^1(\mathbb{T})$ we may find a subsequence $\{S_{m_j}f\}$ of its partial sums such that $m_j\leq M_j$ and $S_{m_j}f \rightarrow f$ almost everywhere in $\mathbb{T}$? This is known as Ul'yanov's problem.   
\end{itemize}

\noindent As for the first question, 
 Hardy's classical inequality \cite{hard13} states that
\begin{align*}
\left|S_nf(\theta)\right|=o(\log n)\text{\  \ as\ \  } n \rightarrow \infty
\end{align*}
Concerning the second question,  Zygmund (\cite{zygmund59}, Ch. 13) proved the following result. 
\begin{theorem} \label{zygmund}
For every $f \in L^1(\mathbb{T})$ and for almost every $\theta \in \mathbb{T}$, there exists a sequence $\{n_k\}$ (depending on~$\theta$) such that $S_{n_k}f(\theta) \rightarrow f(\theta)$ and
\begin{align*}
\frac{\#(\{n_k\} \cap [0,N])}{N} \xrightarrow{N\rightarrow \infty} 1
\end{align*}
\end{theorem}
A partial answer to the third question was found by Konyagin \cite{kon04}: 
\begin{theorem}
There exists a sequence $\{M_j\}$ such that for every function $f \in L^1(\mathbb{T})$ there is an increasing sequence $\{m_j\}$ 
such that $m_j \leq M_j$ for infinitely many $j$ and $S_{m_j}f \rightarrow f$ almost everywhere.
\end{theorem}
Here $M_j$ grows faster than any multiple iteration of the exponent.  A more thorough review of these results appears in Konyagin's survey~\cite{kon06}.  

\bigskip 
In this paper we revisit Zygmund's classical Theorem \ref{zygmund}, our main goal being a higher-dimensional version of his result.
Before going into more details, we remark that Zygmund deduced~Theorem \ref{zygmund} from 
 the following  
asymptotic vanishing of strong arithmetic means of Fourier partial sums:
\begin{theorem}\label{strong sums}
For any $f \in L^1(\mathbb{T})$, 
\begin{align*}
\frac{1}{N} \sum_{n=1}^N \left|S_n f(\theta)-f(\theta)\right|^r \rightarrow 0
\end{align*}
for all $r>0$ and for almost every $\theta \in \mathbb{T}$.
\end{theorem}
In fact, Theorem \ref{strong sums} with a {\em single} $r>0$ suffices in order to deduce Theorem~\ref{zygmund}.  
Zygmund based his proof of Theorem \ref{strong sums} on complex variables. To be more specific, his proof relies on the 
Poisson integral of a function in order to bound the derivative of its analytic extension  to the disc.  This  is used in order 
to exploit a convergence estimate for the Poisson integral of a function which then gives rise to a convergence estimate for $\sum_{n=1}^N |S_nf-f|^r$.  

Theorem~\ref{strong sums} has been extended in various directions. 
Gogoladze \cite{gog88} generalized it to  Orlicz classes: 
\begin{align}\label{strong converge}
\frac{1}{N} \sum_{n=1}^N \Phi\left(S_n f(\theta)-f(\theta)\right) \rightarrow 0
\end{align} 
almost everywhere in $\mathbb{T}$ where $\Phi$ is continuous, positive, convex, and $\log \Phi(u)= O(u/\ln \ln u)$.  Gogoladze's proof is based on Zygmund's technique. 
Rodin \cite{rod90} established the same result as Gogoladze, but found that a certain type of radial maximal function is bounded in $BMO$. This allows him 
to conclude the argument by means of the John-Nirenberg inequality.    Finally, we remark that  Karagulyan~\cite{kar06} 
showed that \eqref{strong converge} fails if 
\begin{align*}
\limsup_{t \rightarrow +\infty}\frac{\log \Phi(t)}{t} = \infty.
\end{align*}
 
 \medskip
The inherent complex variable nature of the aforementioned classical body of work precludes any extensions to higher dimensions.
Our first goal is therefore to  
develop a {\em real-variable argument} leading to Theorem \ref{strong sums} that holds up to a choice of subsequence. We succeed in doing this  for $r\leq 2$, but encounter encounter some difficulties
for powers $r>2$.  

To be more specific, we show that for any $f \in L^1(\mathbb{T})$ and any $\lambda>0$ there exists a set $E \subset \mathbb{T}$  with $|E| \leq \frac{1}{\lambda}$ such that 
\begin{align} \label{mainineq}
\sup_{N\ge 1} \frac{1}{N} \sum_{n=1}^N \int_{\mathbb{T} \setminus E} |S_nf(x)|^2 \,dx \leq C \lambda \|f\|_1^2 
\end{align}
where $C$ is an absolute constant.  Three aspects are crucial about this statement. 
First, we are bounding an average of second moments by the $L^1$ norm of $f$.  
Second, without the removal of $E$ from $\mathbb{T}$, the best we can do is invoke Bernstein's inequality which implies
\begin{align*}
\frac{1}{N} \sum_{n=1}^N \|S_nf(x)\|_2^2  \leq C N \|f\|_1^2.
\end{align*}
And third, we emphasize that $E$ does not depend on $N$. 

Our approach relies on no more than the 
Calder\'{o}n-Zygmund decomposition and a covering lemma for dyadic intervals.  
Thus, in contrast to the aforementioned contributions, harmonic extensions to the disk are avoided completely. 
As expected, the real-variable nature of our approach renders it more flexible. 
It extends to higher dimensions as we now describe. 
Define the Fourier multiplier operators
\begin{align*}
S_{\overline{n}}f= \mathcal{F}^{-1}(\chi_{R_{\overline{n}}}\hat{f}\;)
\end{align*}
where $\chi_{R_{\overline{n}}}$ is the indicator function defined on $\mathbb{Z}^d$ for the rectangle $$R_{\overline{n}}=\{ \overline{m} \in \mathbb{Z}^d : -n_j \leq m_j \leq n_j,\ 1\leq j \leq d\}$$
 Then for all $f \in L^1(\mathbb{T}^d)$ and $\lambda>0$, there exists $E$ such that $\mathscr{H}^d(E) \leq \frac{1}{\lambda}$ 
 (where $\mathscr{H}^d$ is $d$-dimensional Hausdorff measure)
  and 
\begin{align*}
\sup_{N\ge1} \frac{1}{N^d} \sum_{\overline{n} \in R_N^+} \int_{\mathbb{T}^d \setminus E} |S_{\overline{n}}f(x)|^2 \,dx \leq C \lambda \|f\|_1^2
\end{align*}
where 
$$R_N^+=\{ \overline{n} \in (\mathbb{N}\setminus \{0\})^d : \|\overline{n}\|_{\infty} \leq N\}$$ 
and with $C$ an absolute constant.  Then this leads to the following analogue of Theorem~\ref{strong sums}:
\begin{theorem} \label{strong sums high d}
For any $f \in L^1(\mathbb{T}^d)$, there exists a sequence $\{N_k\} \subset \mathbb{N}$
\begin{align*}
\frac{1}{N_k^d} \sum_{\overline{n} \in R_{N_k}^+} \left|S_{\overline{n}} f(\theta)-f(\theta)\right|^r \xrightarrow{k \rightarrow \infty} 0
\end{align*}
 for $0<r\leq 2$.
\end{theorem}
As in Zygmund's case, we may deduce a density statement. 
\begin{corollary} \label{converge high d}
For every $f \in L^1(\mathbb{T}^d)$ there exists a sequence $\{\overline{n}_{\ell}\} \subset \mathbb{N}^d$ (depending on $\theta$) such that $S_{\overline{n}_{\ell}}f(\theta) \rightarrow f(\theta)$ and
\begin{align*}
\limsup_{N \rightarrow \infty}\frac{\#(\{\overline{n}_{\ell}\} \cap [0,N]^d)}{N^d} \xrightarrow{k\rightarrow \infty} 1
\end{align*}
\end{corollary}

Our paper is organized as follows. In the following section we present some preliminary results concerning Calder\'on-Zygmund decompositions
and how they relate to the Ces\`aro-type averages we wish to investigate.  As expected, the slowly decaying tails of  the Dirichlet kernel 
appearing in~\eqref{mainineq} are the source of some technical difficulties.  We isolate the most serious one of these difficulties, and formulate a covering 
lemma which allows us to deal with it.  The covering lemma is proved both on~$\mathbb{T}$ as well as on~$\mathbb{T}^d$.

In Section~3 we present the core of our argument, i.e., the real-variable proof of~\eqref{mainineq}. In addition, we obtain the analogous estimate on the line~$\mathbb{R}$
as well as on higher-dimensional tori~$\mathbb{T}^d$. 

In the final section, we use standard arguments to establish the weak-type estimates on the associated Ces\`aro maximal functions which in turn lead
the desired a.e.~convergence results in any dimension. We also present the density result of Corollary~\ref{converge high d}. 

In higher dimensions it is of course natural to ask about analogues of our averaging theorems relative to other geometries. Most importantly, instead of partial
sums over rectangles we may wish to study partial sums over balls in Fourier space. 
 This will be treated elsewhere. Another unanswered question is whether the argument presented here can be used to obtain the complete statement Theorem \ref{zygmund} and Theorem \ref{strong sums}.  We believe that this is possible but an argument for that is not presented here.
\section{Calder\'{o}n-Zygmund decompositions and bounded Fourier support}
We first present some preliminary results related to our main theorems.  To set the stage, we begin with an immediate corollary of the Calder\'{o}n-Zygmund decomposition:
\begin{proposition} \label{1st reduction}
Let $\lambda>1$.  Then for any $f \in L^1(\mathbb{T})$, there exists $E \subset \mathbb{T}$, with $|E| \leq \frac{1}{\lambda}$ such that 
\begin{align*} 
 \int_{\mathbb{T} \setminus E} |f(\theta)|^2\,d\theta \leq \lambda \|f\|_1^2  
\end{align*}
\end{proposition}
\begin{proof}
Fix $f$ and suppose $\|f\|_1=1$.  We first perform a Calder\'{o}n-Zygmund decomposition at height $\lambda>1$.  Of course, the case $\lambda\leq1$ is trivial because $|\mathbb{T}|=1$.  From the C-Z decomposition we get a collection of disjoint, dyadic intervals, $\mathcal{B}$, such that for any $I \in \mathcal{B}$
\begin{align*}
\lambda< \frac{\int_I |f|}{|I|} \leq 2\lambda
\end{align*}
and for any $x \in \mathbb{T} \setminus \cup_{I \in \mathcal{B}} I$, $|f(x)| \leq \lambda$.  Then we decompose $f$ as $f=g+b$ where
\begin{align*}
b:= \sum_{I \in \mathcal{B}} \chi_I f
\end{align*}
and $g:=f-b$.  Then we can let 
\begin{align*}
E:=\bigcup_{I \in \mathcal{B}} I.
\end{align*}
From this we have $|E|\leq \frac{1}{\lambda}$ and
\begin{align*}
\int_{\mathbb{T} \setminus E} |f(\theta)|^2\,d\theta &= \int_{\mathbb{T} \setminus E} |g(\theta)|^2\,d\theta \\
&\leq \int_{\mathbb{T} \setminus E} \lambda|g(\theta)|\,d\theta \leq \lambda\int_{\mathbb{T}} |g(\theta)|\,d\theta \\
&\leq \lambda\|g\|_1 \leq \lambda
\end{align*}
Then for any $f \in L^1$, let $h:= f/\|f\|_1$ and we have
\begin{align*}
 &\int_{\mathbb{T} \setminus E} |h(\theta)|^2\,d\theta \leq \lambda \Rightarrow \int_{\mathbb{T} \setminus E} \left|\frac{f(\theta)}{\|f\|_1}\right|^2\,d\theta \leq \lambda \\
&\Rightarrow \int_{\mathbb{T} \setminus E} |f(\theta)|^2\,d\theta \leq \lambda \|f\|_1^2.
\end{align*}
This completes the proof.
\end{proof}
Next, we show that the previous estimate remains essentially unchanged if we introduce an "uncertainty" of scale $\frac{1}{N}$ into the function $f$.  This is the first major step toward our main result.  We note that in this proposition, $f$ depends on $N$ and thus the exceptional set $E$ implicitly depends on $N$.  Due to this circumstance, our main theorem on Fourier series is not an immediate corollary of the following proposition. Rather, we shall need to rely on the covering lemma which is presented later in this section, see Lemma \ref{lemma}.
\begin{proposition} \label{2nd reduction}
Let $N \in \mathbb{N}$, $\lambda>0$.  Then for any $f \in L^1(\mathbb{T})$ such that $\mbox{supp}(\hat{f}) \subset [-N,N]$, there exists $E \subset \mathbb{T}$, with $|E| \lesssim \frac{1}{\lambda}$ such that 
\begin{align*}
\int_{\mathbb{T} \setminus E} (B_{N}*|f|^2)(\theta)\,d\theta \leq C \lambda \|f\|_1^2  
\end{align*}
where $B_N(x)= \frac{1}{N}\chi_{[-1/2N, 1/2N]}(x)$.
\end{proposition}
\begin{proof}
We assume again that $\|f\|_1=1$.  Fix $\lambda>1$ and $N \in \mathbb{N}$.  We still perform a Calder\'{o}n-Zygmund at height $\lambda$ for $f$, and denote the collection of "bad" intervals given by the decomposition by $\mathcal{B}$. But we note that we cannot expect $B_N*|f|^2 \leq \lambda^2$ on $\mathbb{T} \setminus \cup_{I \in \mathcal{B}}I$ as in the previous problem due to the "smearing" of the function $f$ from the convolution with $B_N$.  We begin to handle this problem by simply defining 
\begin{align*}
E := \bigcup_{I \in \mathcal{B}} 3 \cdot I
\end{align*}
where $3 \cdot I$ is the interval with the same center as the interval $I$ but which is of length $3|I|$. We will split $f$ into three parts and show that the inequality holds for each part. From the Calder\'{o}n-Zygmund decomposition, we have $f=g+b$.  We split the collection $\mathcal{B}$ into $\mathcal{B}_1$ and $\mathcal{B}_2$, where $\mathcal{B}_1$ is the collection of intervals in $\mathcal{B}$ with length greater than $1/N$ and $\mathcal{B}_2$ is the collection of intervals in $\mathcal{B}$ with length less than or equal to $1/N$. Then, let $b=b_1+b_2$ and $f_I := \chi_I f$ where
\begin{align*}
b_1= \sum_{I \in \mathcal{B}_1} \chi_I f= \sum_{I \in \mathcal{B}_1} f_I \\
\mbox{and } b_2= \sum_{I \in \mathcal{B}_2} \chi_I f= \sum_{I \in \mathcal{B}_2} f_I.
\end{align*}
Thus $f=g+b_1+b_2$.  Let $V_N$ be the de la Vall\'{e}e Poussin kernel.  Since $\mbox{supp}(\hat{f}) \subset [-N,N]$, $f=V_N*f=:f^{(N)}$ and thus $f^{(N)}=g^{(N)}+b^{(N)}_1+b^{(N)}_2$. The sharp cutoff's of the indicator functions introduce higher order Fourier coefficients, so $\mbox{supp}(\hat{g})$, $\mbox{supp}(\widehat{b_1})$, and  $\mbox{supp}(\widehat{b_2})$ are no longer necessarily contained in $[-N, N]$.  This prevents us from having $g^{(N)}=g$, $b^{(N)}_1=b_1$, and $b^{(N)}_2=b_2$.  Although $g^{(N)}$, $b^{(N)}_1$, and $b^{(N)}_2$ are not supported on disjoint sets, we will allow ourselves to gain a constant factor in order to bound each term individually:
\begin{align*}
&\int_{\mathbb{T} \setminus E} B_{N}*|f|^2 \lesssim \int_{\mathbb{T} \setminus E} B_{N}*|g^{(N)}|^2 +\int_{\mathbb{T} \setminus E} B_{N}*|b_1^{(N)}|^2 +\int_{\mathbb{T} \setminus E} B_{N}*|b^{(N)}_2|^2
\end{align*}
We first consider $g^{(N)}$, for which we have 
\begin{align*}
\|g^{(N)}\|_{\infty}= \|g * V_N\|_{\infty} \leq \|g\|_{\infty}\|V_N\|_1 \leq \lambda \|V_N\|_1 \lesssim\lambda
\end{align*}
whence
\begin{align*}
\int_{\mathbb{T} \setminus E} (B_{N}*|g^{(N)}|^2)(\theta)\,d\theta &\lesssim \int_{\mathbb{T} \setminus E} \lambda (B_{N}*|g^{(N)}|)(\theta)\,d\theta \\
&\lesssim \int_{\mathbb{T}}\lambda (B_{N}*|g^{(N)}|)(\theta)\,d\theta =\lambda \|B_N * |g^{(N)}|\|_1 \\
& \lesssim \lambda \|B_N\|_1 \|g^{(N)}\|_1 \lesssim \lambda \|g\|_1\|V_N\|_1\\
& \lesssim \lambda
\end{align*}
which is all we need to show for $g$.  For $b^{(N)}_1$, let $I^*:=2\cdot I$ and $V_N*b_1=b^*+\tilde{b}$, where 
\begin{align*}
b^*=\sum_{I \in \mathcal{B}_1}\chi_{I^*}(V_N*f_I) \mbox{ and }\tilde{b}=\sum_{I \in \mathcal{B}_1}\chi_{\mathbb{T}\setminus  I^*}(V_N*f_I).
\end{align*}
Then 
\begin{align*}
\int_{\mathbb{T} \setminus E} B_{N}*|b_1^{(N)}|^2 \lesssim \int_{\mathbb{T} \setminus E} B_{N}*|b^*|^2+\int_{\mathbb{T} \setminus E} B_{N}*|\tilde{b}|^2
\end{align*}
We can handle $b^*$ easily.  By definition of $b^*$,  $\mbox{supp}(|b^*|^2)= \cup_{I \in \mathcal{B}_1} 2\cdot I$ and thus
\begin{align}\label{setineq}
\mbox{supp}(B_{N}*|b^*|^2)= \bigcup_{I \in \mathcal{B}_1} 2\cdot I+\left[-\frac{1}{2N}, \frac{1}{2N}\right] \subset \bigcup_{I \in \mathcal{B}_1} 3\cdot I=E 
\end{align}
The set inequality from the previous line is due to the fact that $|I|> \frac{1}{N}$ for all $I \in \mathcal{B}_1$.  This implies 
\begin{align*}
\int_{\mathbb{T} \setminus E} B_{N}*|b^*|^2=0
\end{align*}
In order to handle $\tilde{b}$ we start with a claim:

\medskip
{\bf Claim: }$\|\tilde{b}\|_{\infty}=\|\sum_{I}\chi_{\mathbb{T}\setminus I^*}f^{(N)}_I\|_{\infty} \lesssim \lambda$.
\medskip

\noindent We consider the two possible cases for any $x \in \mathbb{T}$: either (1) $x \in \mathbb{T} \setminus \cup I^*$ or (2) $x \in I^*$ for at least one $I \in \mathcal{B}_1.$ In the first case, for any $x \in \mathbb{T} \setminus \cup I^*$, $\chi_{\mathbb{T}\setminus I^*}(x)=1$ so
\begin{align*}
\left|\tilde{b}(x)\right|&\lesssim \sum_{I \in \mathcal{B}_1} \chi_{\mathbb{T}\setminus I^*}(x)|f^{(N)}_I(x)|=\sum_{I \in \mathcal{B}_1} \chi_{\mathbb{T}\setminus I^*}(x)|(V_N*f_I)(x)|=\sum_{I \in \mathcal{B}_1}|(V_N*f_I)(x)| \\
&\lesssim \sum_{I \in \mathcal{B}_1}(|V_N|*|f_I|)(x)
\end{align*}
In the second case, $x \in H_k^*$ for some subcollection of $\mathcal{B}_1$, $\{H_k\}$. Then  $\chi_{\mathbb{T}\setminus I^*}(x)=1$ for $I^* \not\in \{H_k\}$ and $\chi_{\mathbb{T}\setminus H_k^*}(x)=0$, so 
\begin{align*}
\tilde{b}(x)&=\sum_{I \in \mathcal{B}_1} \chi_{\mathbb{T}\setminus I^*}(x)|f^{(N)}_I(x)|=\sum_{I \in \mathcal{B}_1 \atop I \not\in \{H_k\}} \chi_{\mathbb{T}\setminus I^*}(x)|f^{(N)}_I(x)| \\
&\leq \sum_{I \in \mathcal{B}_1 \atop I \not\in \{H_k\}} (|V_N|* |f_I|)(x).
\end{align*}
In this case, for every $|V_N|*|f_I|$ in the sum, $x \in \mathbb{T} \setminus I^*$. So in both cases, we are taking a sum of $(|V_N|*|f_I|)(x)$ where $x \not\in \cup I^*$ and the union is taken over the same intervals as the sum.  Therefore, it suffices to assume that we are in the first case and $x \in \mathbb{T} \setminus \cup I^*$, where the union is taken over all $I \in \mathcal{B}_1$.  So we fix some $x  \in \mathbb{T} \setminus \cup I^*$.  First in order to bound each $|V_N|* |f_I|$, we recall that $|V_N(y)| \lesssim \frac{1}{N} \mbox{min}(N^2, |y|^{-2})$, then
\begin{align*}
|\tilde{b}(x)| &\lesssim \sum_{I \in \mathcal{B}_1}(|V_N|*|f_I|)(x) = \sum_{I \in \mathcal{B}_1} \int_{\mathbb{T}} |V_N(x-y)| |f_I(y)|\,dy \\
&\lesssim \int_{\mathbb{T}} \frac{1}{N}\mbox{min}(N^2,|x-y|^{-2})\sum_{I \in \mathcal{B}_1} |f_I(y)|\,dy
\end{align*}
The $|f_I(y)|$ gives us that the product is nonzero for $y \in I$, and by assumption $x \in \mathbb{T} \setminus I^*$.  Therefore, $|x-y|> \frac{1}{2}|I|>\frac{1}{2N}$ (for every $I \in \mathcal{B}_1$) which implies $|V_N(x-y)| \lesssim \frac{1}{N}|x-y|^{-2}$ and 
\begin{align*}
&\int_{\mathbb{T}} \frac{1}{N}\mbox{min}(N^2,|x-y|^{-2})\sum_{I \in \mathcal{B}_1} |f_I(y)|\,dy \\
&\lesssim \frac{1}{N}\sum_{I \in \mathcal{B}_1} \sup_{y \in I}|x_0-y|^{-2} \|f_I\|_1 & \mbox{By H\"{o}lder}
\end{align*}
for any fixed $x \in \mathbb{T}\setminus \cup I^*$.  For each $I$, $|x-y|$ with $y \in I$ is at least $\frac{1}{2}|I|$.  Of course, $|x-y|$ will be greater than the distance between $x$ and $I$ for any $y \in I$. Then $|x-y|\gtrsim \max(|I|, \mbox{dist}(x,I)) $ which implies $|x-y| \gtrsim |I|+\mbox{dist}(x,I) $.  Thus by ordering $\mathcal{B}_1=\{I_j\}$ by proximity to $x$ (considering only those intervals to the right of $x$ without loss of generality) we obtain
\begin{align*}
&\frac{1}{N}\sum_{I \in \mathcal{B}_1} \sup_{y \in I}|x-y|^{-2} \|f_I\|_1 \\
&\lesssim \frac{1}{N} \sum_{j =1}^{|\mathcal{B}_1|} (|I_j|+\mbox{dist}(x,I_j))^{-2}\|f_{I_j}\|_1 \\
&\lesssim \frac{\lambda}{N} \sum_{j=1}^{|\mathcal{B}_1|} \frac{|I_j|}{(|I_j|+\mbox{dist}(x,I_j))^{2}} & \mbox{By the C-Z decomposition.}
\end{align*}
Let $\phi_x(y):=\min(N^{2}, |x-y|^{-2})$.  Then, for all $y \in I_j$,
\begin{align*}
\frac{1}{(|I_j|+\mbox{dist}(x,I_j))^{2}} &\lesssim \frac{1}{|x-y|^{2}}   \lesssim \phi_x(y) \\
\Rightarrow  \frac{1}{(|I_j|+\mbox{dist}(x,I_j))^{2}} &\lesssim \inf_{y \in I_j}\phi_x(y)
\end{align*}
We also recall that the $I_j$ are pairwise disjoint. Therefore, the sum $\sum_{j=1}^{|\mathcal{B}_1|} \frac{|I_j|}{(|I_j|+\mbox{dist}(x,I_j))^{2}}$ is bounded by a lower Riemann sum of $\phi_x(y)$.  Whence,
\begin{align*}
\frac{1}{N} \sum_{j=1}^{|\mathcal{B}_1|} \lambda\frac{|I_j|}{(|I_j|+\mbox{dist}(x,I_j))^{2}} &\lesssim \frac{\lambda}{N} \|\phi_x\|_1 \\
&\lesssim  \frac{\lambda}{N} N =\lambda.
\end{align*}
In conclusion, the claim holds and $\|\tilde{b}\|_{\infty} \lesssim \lambda$, which implies
\begin{align*}
\int_{\mathbb{T}\setminus E} (B_N*|\tilde{b}|^2)(x) \,dx &\lesssim \|B_N\|_1 \int_{\mathbb{T}\setminus E} |\tilde{b}(x)|^2 \,dx  \\
&\lesssim \|B_N\|_1 \lambda \int_{\mathbb{T}\setminus E} |\tilde{b}(x)| \,dx \hspace{1cm} (\mbox{from } \|\tilde{b}\|_{\infty} \lesssim \lambda)\\
&\lesssim  \lambda \sum_{I \in \mathcal{B}_1}\int_{\mathbb{T}}\chi_{\mathbb{T}\setminus I^*}(x)|f_I^{(N)}(x)|\,dx \lesssim  \lambda\sum_{I \in \mathcal{B}_1}\int_{\mathbb{T}}|f_I^{(N)}(x)|\,dx\\ 
&\lesssim \lambda \sum_{I \in \mathcal{B}_1} \|f_I\|_1 \lesssim \lambda.
\end{align*}
Finally we have shown for $b^{(N)}_1$
\begin{align*}
\int_{\mathbb{T}\setminus E} B_N*|b^{(N)}_1|^2 \lesssim \int_{\mathbb{T}\setminus E} B_N*|b^*|^2+\int_{\mathbb{T}\setminus E} B_N*|\tilde{b}|^2\lesssim \lambda.
\end{align*}
  For $b_2$, we first assume $N=2^j$ for some $j \in \mathbb{N}$ and we let $J$ be the collection of dyadic intervals of length $N^{-1}=2^{-j}$.  We will basically show that there should not be any bad intervals of length less than or equal to $1/N$.  Heuristically speaking, we expect this since by the Uncertainty Principle, functions with Fourier support contained in $[-N, N]$ are essentially constant on the scale of $\frac{1}{N}$. To be more precise we shall now establish the bound $\|b_2^{(N)}\|_{\infty} \lesssim \lambda$.  Since $I \in \mathcal{B}_2$ are dyadic, each $I$ is contained in a length $\frac{1}{N}$ interval.  So we let
\begin{align*}
b_2^{(N)}&=\sum_{I \in \mathcal{B}_2}f^{(N)}_I= \sum_{J} \sum_{I \in \mathcal{B}_2 \atop I \subset J}f_I^{(N)} 
\end{align*}
where the $J$ intervals come from the $1/N$ partition. Henceforth, we will let 
\begin{align*}
f_J:= \chi_J b_2= \sum_{I \in \mathcal{B}_2 \atop I \subset J}f_I.
\end{align*}
 We first note again that $|V_N(\theta)| \lesssim \frac{1}{N} \min(|\theta|^{-2}, N^2)$ and thus 
\begin{align*}
|f^{(N)}_J(y)| &\lesssim \int_{\mathbb{T}}\frac{1}{N}\min(|y-x|^{-2},N^2)|f_J(x)| \,dx \\
& \lesssim \frac{1}{N}\|f_J\|_1  \min(\sup_{x \in J}|y-x|^{-2},N^2)\\
& \lesssim \frac{1}{N}\|f_J\|_1 \frac{1}{\mbox{dist}(y,J)^2+\frac{1}{N^2}}
\end{align*}
Then
\begin{align*}
\left|b^{(N)}_2(y)\right|&=\left| \sum_J f^{(N)}_J(y) \right| \lesssim  \sum_J \left|f^{(N)}_J(y)\right|\\
& \lesssim \sum_J \|f_J\|_1 \frac{1/N}{\mbox{dist}(y, J)^2+\frac{1}{N^2}} \\
&\lesssim (\max_J \|f_J\|_1) \sum_J  \frac{1/N}{\mbox{dist}(y, J)^2+\frac{1}{N^2}}
\end{align*}
We also note that by the Calder\'{o}n-Zygmund decomposition for any $J$,
\begin{align*}
\|f_J\|_1&=\frac{1}{N}\frac{\|f_J\|_1}{|J|} \lesssim \frac{1}{N}\frac{\sum_{I \in \mathcal{B}_2 \atop I \subset J} \|f_I\|_1}{|J|} \\
& \lesssim  \frac{1}{N}\frac{\lambda \sum_{I \in \mathcal{B}_2 \atop I \subset J} |I|}{|J|} \lesssim \frac{1}{N}\frac{\lambda |J|}{|J|}\\
&\lesssim \lambda|J| =\frac{\lambda}{N}
\end{align*}
Therefore,
\begin{align*}
\left|b^{(N)}_2(y)\right|&\lesssim (\max_J \|f_J\|_1)\sum_J \frac{1/N}{\mbox{dist}(y, J)^2+\frac{1}{N^2}} \\
&\lesssim \lambda\sum_J  \frac{1/N^2}{\mbox{dist}(y, J)^2+\frac{1}{N^2}}
\end{align*}
If we fix $y$, then for each $J$, there is a nonnegative constant $C_y \leq \frac{1}{N}$ that does not depend on $J$ and a positive integer $m_J \in [1,N]$ unique to each $J$ such that $\mbox{dist}(y,J) = C_y+m_J\frac{1}{N} $
\begin{align*} 
\sum_J  \frac{1/N^2}{\mbox{dist}(y, J)^2+\frac{1}{N^2}} &\lesssim \sum_J  \frac{1/N^2}{(C_y+m_J\frac{1}{N})^2+\frac{1}{N^2}}\lesssim \sum_J  \frac{1/N^2}{m^2_J\frac{1}{N^2}+\frac{1}{N^2}} \\
&\lesssim \sum_{i=1}^N \frac{1}{i^2} \lesssim 1 
\end{align*}
This gives us the following inequality for all $y \in \mathbb{T}$:
\begin{align*}
\left|b^{(N)}_2(y)\right|=\left| \sum_J f^{(N)}_J(y) \right| \lesssim \lambda.
\end{align*}
Thus 
\begin{align*}
\int_{\mathbb{T} \setminus E} B_N * \left| b^{(N)}_2 \right|^2 &\lesssim \lambda \int_{\mathbb{T}} B_N * \left|b^{(N)}_2\right|= \lambda \|B_N * |b_2^{(N)}|\|_1 \\
&\lesssim \lambda  \|B_N\|_1  \|b_2^{(N)}\|_1 \hspace{.3cm} \mbox{By Young's Inequality} \\
&\lesssim \lambda \|b_2\|_1\lesssim \lambda\|f\|_1\lesssim \lambda \end{align*}
Which is exactly what we need for $b_2^{(N)}$.

Combining the three bounds, we have
\begin{align*}
\int_{\mathbb{T} \setminus E} (B_{N}*|f|^2)(\theta)\,d\theta &=\int_{\mathbb{T} \setminus E} (B_{N}*|f^{(N)}|^2)(\theta)\,d\theta \\
&\lesssim \int_{\mathbb{T} \setminus E} (B_{N}*|g^{(N)}|^2)(\theta)\,d\theta \\
&\hspace{.5cm}+\int_{\mathbb{T} \setminus E} (B_{N}*|b^{(N)}_1|^2)(\theta)\,d\theta +\int_{\mathbb{T} \setminus E} (B_{N}*|b^{(N)}_2|^2)(\theta)\,d\theta \\
&\lesssim \lambda +\lambda +\lambda \lesssim \lambda.
\end{align*}
By scaling we obtain the final result without assuming $\|f\|_1=1$ and by taking the Calder\'{o}n-Zygmund decomposition at height $\frac{1}{3}\lambda$ we can assume that $|E|\leq \frac{1}{\lambda}$.
\end{proof}
It is now natural, as well as essential for our main application to Fourier series later in the paper, to generalize Proposition \ref{2nd reduction} to kernels other than $B_N$.  Technically speaking, we lose the simple relation \eqref{setineq}, i.e.,
\begin{align*}
\mbox{supp}(B_{N}*|b^*|^2)= \bigcup_{I \in \mathcal{B}_1} 2\cdot I+\left[-\frac{1}{2N}, \frac{1}{2N}\right] \subset \bigcup_{I \in \mathcal{B}_1} 3\cdot I.
\end{align*}
Suppose that instead of $B_N$ we had the kernel $Q_N$ where
\begin{align*}
Q_N(y) = \frac{1}{N^{s-1}} \min(N^s, |y|^{-s}), \hspace{.5cm} s>0
\end{align*}
Of course, we need to assume $s>1$. Indeed, for $0<s\leq1$, the $L^1$ norm of $Q_N$ is not bounded as $N\rightarrow \infty$; in fact for $f\equiv 1$ we have
\begin{align*}
\int_{\mathbb{T} \setminus E} (Q_N*|f|^2)(x)\,dx\gtrsim (1-\frac{1}{\lambda})\|Q_N\|_1.
\end{align*}
Thus, $(1-\frac{1}{\lambda})\|Q_N\|_1$ will be unbounded as $N \rightarrow \infty$. Thus, we assume $s>1$.  Of particular importance later is the case $s=2$. Focusing on the aforementioned relation \eqref{setineq} we face the issue of bounding
\begin{align*}
\int_{\mathbb{T} \setminus E} (Q_N*|b^*|^2)(x)\,dx.
\end{align*}
An essential problem that we face in this situation is possible overlap that accumulates in the sum $\sum_{I \in \mathcal{B}_1} \chi_{I^*}f^{(N)}_I$.  In addition to this obstruction, we again emphasize that the hypothesis $\mbox{supp}(\hat{f}) \subset [-N,N]$ from Proposition \ref{2nd reduction} gives $E$ an implicit dependence on $N$.  The following covering lemma is designed to handle both of these obstructions.
\begin{lemma} \label{lemma}
Let $\mathcal{G}$ be a finite collection of pairwise disjoint, nonadjacent dyadic intervals.  Let $\mathcal{G}^*$ be the collection of dilated intervals that are each of the form $\frac{9}{8}\cdot I=:I^*$ for all $I \in \mathcal{G}$.  If $\cup_{J \in \mathcal{G}^*}J$ is an interval, then 
\begin{align}   \bigcup_{J \in \mathcal{G}^*} J \subset 4\cdot J_0 \end{align}
where $J_0$ is the largest interval in $\mathcal{G}$.
\end{lemma}
\begin{proof} We begin with a definition.  Consider the intervals $I_1$, $I_2$, and $I_3$ in $\mathcal{G}$ where the order corresponds to their placement on the real line from left to right.  Assuming that $I^*_1\cup I^*_2\cup I_3^*$ is an interval and that $I^*_1 \cap I^*_3=\emptyset$, we call $\{I^*_1, I_2^*, I_3^*\}$ a {\bf \em chain} and $I^*_2$ the {\bf \em bridge}.

{\em Claim:} For any chain $\{I_1^*,I_2^*, I_3^*\}$ the bridge $I_2^*$ is strictly bigger than at least one of either $I_1^*$ or $I_3^*$.  In other words
\begin{align*}
|I_2^*|> \mbox{min}(|I_1^*|, |I_3^*|).
\end{align*}
Let us assume $2^k=|I_1|\leq |I_3|=2^{\ell}$.  Since $I_1$ and $I_3$ are dyadic intervals and are nonadjacent, one has $\mbox{dist}(I_1,I_3)=n\cdot 2^k$ for some integer $n \geq 1$.  $\frac{1}{16}2^k$ is the distance between the right endpoint of $I_1$ and the right endpoint of $I^*_1$ and similarly $\frac{1}{16}2^{\ell}$ is the same for $I_3$ and $I^*_3$. Because of $I_1^* \cap I_3^* = \emptyset$, we have $n \cdot 2^k \geq \frac{1}{16}(2^k+2^{\ell})$, whence
\begin{align*}
&n \cdot 2^k \geq \frac{1}{16}(2^k+2^{\ell}) \Rightarrow (16n-1)2^k \geq 2^{\ell}\Rightarrow 16n-1 \geq 2^{\ell-k}\\ 
& \Rightarrow n \geq 2^{\ell-k-4}+\frac{1}{16} \\
&\mbox{If } \ell \geq k+4 \mbox{ then, since $n \in \mathbb{Z}$, }  n \geq 2^{\ell-k-4}+\frac{1}{16}  \Rightarrow n\geq 2^{\ell-k-4}+1\\
&\mbox{If } k \leq \ell \leq k+3 \mbox{ then }  n \geq 2^{\ell-k-4}+\frac{1}{16} \Rightarrow n\geq 1
\end{align*}
It is obvious that $I^*_2$ must cross the gap between $I_1^*$ and $I_3^*$ where the length of the gap can be calculated as $n\cdot 2^k- \frac{1}{16}(2^k+2^{\ell})$. If $\ell\geq k+4$ then
\begin{align*}
\mbox{length of gap} &= n\cdot 2^k- \frac{1}{16}(2^k+2^{\ell})\geq \left(2^{\ell-k-4}+1\right)2^k- \frac{1}{16}(2^k+2^{\ell})\\
&=2^{\ell-4}+2^k-\frac{1}{16}2^k-\frac{1}{16}2^{\ell}\\
&=\frac{15}{16}2^k
\end{align*}
Thus, in this case, the length of the gap is strictly greater than $\frac{9}{8}2^{k-1} = \frac{9}{16}2^k$ and thus $I^*_2$ must be as long as $I_1^*$.  In the case, $k \leq \ell \leq k+3$
\begin{align*}
\mbox{length of gap} &= n\cdot 2^k- \frac{1}{16}(2^k+2^{\ell})\geq 2^k- \frac{1}{16}(2^k+2^{k+3})\\
&=\frac{7}{16}2^k
\end{align*}
Therefore,  the length of the gap is strictly greater than $\frac{9}{8}2^{k-2} = \frac{9}{32}2^k$ and thus $I^*_2$ must be as long as $\frac{1}{2}|I_1^*|$.  At this point, we have shown that $|I^*_2| \geq \frac{1}{2}|I_1^*|$ in all cases. It remains to show that $|I^*_2|$ equals neither  $|I_1^*|$ nor $\frac{1}{2}|I_1^*|$. Of course, in all cases we must have $I_2^* \cap I_1^*\neq \emptyset$ if $I_2^*$ were to be a suitable bridge between $I^*_1$ and $I_3^*$, and the claim is that if  $|I^*_2| = |I_1^*|$ or $\frac{1}{2}|I_1^*|$ then  $I_2^* \cap I_1^*= \emptyset$.  With $m=k-i$, for $i=1,0$, $I_2$ and $I_1$ being non adjacent implies $\mbox{dist}(I_2,I_1)\geq 2^m$. Then
\begin{align*}
\mbox{dist}(I^*_1,I_2^*) &\geq 2^m-\frac{1}{16}(2^k+2^m)\\
&=15\cdot 2^{m-4}-2^{k-4}\\
&\geq 15\cdot 2^{k-5}-2^{k-4}\\
&=13\cdot 2^{k-5}>0
\end{align*}
Hence $I_2^* \cap I_1^*= \emptyset$ if $|I^*_2| = |I_1^*|$ or $\frac{1}{2}|I_1^*|$ and thus $|I^*_2|>|I^*_1|$.

Therefore, we have shown that for any chain $\{I^*_1,I^*_2,I^*_3\} \subset \mathcal{G}^*$, $|I^*_2| > \mbox{min}(|I^*_1|,|I^*_3|)$.  Let $\mathcal{G}^*=\{J_0, J_1, ..., J_n\}$ where the dilated intervals are ordered by length, and, for each $0\leq i\leq n$, let $J_i=[a_i,b_i)$. If each $J_i \in \mathcal{G}^*$ is contained in $J_0$ then we are done and $\cup_i J_i \subset 4 \cdot J_0$.  If not, let $J_k$ be the dilated interval that extends the furthest to the right of $J_0$ in $\mathcal{G}^*$ such that $J_0\cap J_k \neq \emptyset$ and $J_k\not\subset J_0$.  In the conclusion of the proof of the claim, part of what we showed was that if $|J_0|=|J_k|$ then $J_0 \cap J_k =\emptyset$.  We will first show that $4 \cdot J_k \subset 4 \cdot J_0$.  Note that
\begin{align*}
J_0&=[a_0, b_0) = \left[ \frac{a_0+b_0}{2}-\frac{b_0-a_0}{2}, \frac{a_0+b_0}{2}+ \frac{b_0-a_0}{2} \right) \\
4 \cdot J_0 &= \left[ \frac{a_0+b_0}{2}-4\frac{b_0-a_0}{2}, \frac{a_0+b_0}{2}+ 4\frac{b_0-a_0}{2} \right)\\
 J_k&=[a_k, b_k) = \left[ \frac{a_k+b_k}{2}-\frac{b_k-a_k}{2}, \frac{a_k+b_k}{2}+ \frac{b_k-a_k}{2} \right) \\
4 \cdot J_k &= \left[ \frac{a_k+b_k}{2}-4\frac{b_k-a_k}{2}, \frac{a_k+b_k}{2}+ 4\frac{b_k-a_k}{2} \right)
\end{align*}
and thus our goal is to show that 
\begin{align}\label{dilate ineq}
\frac{a_0+b_0}{2}-4\frac{b_0-a_0}{2} &\leq \frac{a_k+b_k}{2}-4\frac{b_k-a_k}{2} \\
\label{dilate ineq 2}\mbox{ and } \frac{a_k+b_k}{2}+ 4\frac{b_k-a_k}{2} &\leq  \frac{a_0+b_0}{2}+ 4\frac{b_0-a_0}{2}.
\end{align}
Inequality \eqref{dilate ineq} is obvious since we assumed $J_k$ is to the right of $J_0$ and smaller than $J_0$ (i.e. $b_0-a_0> b_k-a_k$), which implies
\begin{align*}
\frac{a_0+b_0}{2}-\frac{b_0-a_0}{2} &< \frac{a_k+b_k}{2}-\frac{b_k-a_k}{2} \\
\mbox{and } -\frac{3}{2}(b_0-a_0) &< -\frac{3}{2}(b_k-a_k) \\
\Rightarrow \frac{a_0+b_0}{2}-4\frac{b_0-a_0}{2} &< \frac{a_k+b_k}{2}-4\frac{b_k-a_k}{2}
\end{align*}
In order to show inequality \eqref{dilate ineq 2}, we will use the inequalities $a_k < b_0$ and $|J_k|=b_k-a_k \leq \frac{1}{2}(b_0-a_0)=\frac{1}{2}|J_0|$ which hold by assumption. First, $a_k<b_0$ implies 
\begin{align*}
b_0>a_k \Leftrightarrow \frac{b_0}{2}+\frac{b_0}{2} >\frac{a_k}{2}- \frac{-a_k}{2} &\Leftrightarrow \frac{b_0-a_0}{2}+\frac{b_0+a_0}{2} >\frac{b_k+a_k}{2}- \frac{b_k-a_k}{2}\\
&\Leftrightarrow  \frac{b_0-a_0}{2}+ \frac{b_k-a_k}{2} >\frac{b_k+a_k}{2}-\frac{b_0+a_0}{2}>0.
\end{align*}
The positivity in the second line is due to $J_k$ being to the right of $J_0$.  This shows us that the distance between the centers of the dilated intervals is less than the sum of their lengths halved. So proceeding with showing inequality \eqref{dilate ineq 2} we have
\begin{align*}
&\frac{b_0-a_0}{2}+ \frac{b_k-a_k}{2} >\frac{b_k+a_k}{2}-\frac{b_0+a_0}{2} \Leftrightarrow  \frac{b_0-a_0}{2}+\frac{b_0+a_0}{2} >\frac{b_k+a_k}{2}- \frac{b_k-a_k}{2} \\
&\Leftrightarrow  4\frac{b_0-a_0}{2}+\frac{b_0+a_0}{2} >\frac{b_k+a_k}{2}- \frac{b_k-a_k}{2} +3\frac{b_0-a_0}{2} \\
&\Rightarrow 4\frac{b_0-a_0}{2}+\frac{b_0+a_0}{2} >\frac{b_k+a_k}{2}- \frac{b_k-a_k}{2} +6\frac{b_k-a_k}{2} \hspace{.3cm}\left(b_k-a_k \leq \frac{1}{2}(b_0-a_0) \right) \\
&\Leftrightarrow 4\frac{b_0-a_0}{2}+\frac{b_0+a_0}{2} >\frac{b_k+a_k}{2}+5\frac{b_k-a_k}{2} \\
&\Rightarrow 4\frac{b_0-a_0}{2}+\frac{b_0+a_0}{2} >\frac{b_k+a_k}{2}+4\frac{b_k-a_k}{2}.
\end{align*}
Thus, for $J_k$, $4 \cdot J_k \subset 4 \cdot J_0$  and the same holds for all intervals $J_{k_{\ell}}$ such that $J_{k_{\ell}} \not\subset J_0$ and $J_{k_{\ell}} \cap J_0 \neq \emptyset$.  

Now it is possible that all intervals to the right of $J_0$ are contained in $J_0\cup J_k$, in which case we stop this process.  Otherwise, take the largest dilated interval, $J_m$, extending to the right of $J_k$ such that $J_m \cap J_{k} \neq \emptyset$.  By the maximality of $J_k$, $J_m \cap J_0 = \emptyset$. In this case we have the exact assumptions used in the claim with $J_{k}$ as the middle dilated interval.  This implies $|J_m|< |J_{k}|$ because we already know that $|J_0|>|J_{k}|$.  Thus the argument showing $4 \cdot J_{k} \subset 4 \cdot J_0$ is suitable to show that $4 \cdot J_m \subset 4 \cdot J_{k}$ and $4 \cdot J_{m} \subset 4 \cdot J_0$.

Again if each $J_i \in \mathcal{G}^*$ that is to the right of $J_0$ is contained in $J_0 \cup J_{k} \cup J_{m}$, then we stop this algorithm.  Otherwise the process continues and, at each step, we obtain the same relation as we have between $J_{k}$, $J_{m}$ and $J_0$: 
\begin{align*}
 4 \cdot J_{m} \subset  4 \cdot J_{k}  \subset 4 \cdot J_0 \Rightarrow J_{m} \subset 4 \cdot J_0
\end{align*}
We can use a similar argument for the intervals to the left of $J_0$. Since there are only finitely many $J_i$, this process must exhaust the collection of all $J_i$.
\end{proof}
The constant $\frac{9}{8}$ is not the only admissible factor by which one can increase the size of the intervals, but it is the largest constant of the form $\frac{2^j+1}{2^j}$ for which the lemma holds.  Therefore, $\frac{9}{8}$ is chosen partly by necessity and partly for convenience since the intervals we are considering are all dyadic.  Next, we establish the analogous statement to Lemma \ref{lemma} in higher dimensions.
\begin{lemma} \label{lemma2}
Let $\mathcal{G}$ be a finite collection of pairwise disjoint, nonadjacent (distance between any two is nonzero) dyadic cubes.  Let $\mathcal{G}^*$ be the collection of dilated cubes that are each of the form $\frac{9}{8}\cdot Q$ for all $Q \in \mathcal{G}$.   If $\cup_{H \in \mathcal{G}^*}H$ is a connected set, then  
\begin{align}   \bigcup_{H \in \mathcal{G}^*} H \subset 4\cdot H_0 \end{align}
where $H_0$ is the largest cube in $\mathcal{G}^*$.
\end{lemma}
\begin{proof}
Let $\mathscr{H}^d$ denote the Hausdorff measure of dimension $d$. First note that for a cube $Q$, $\frac{9}{8}\cdot Q$ is the cube with the same center as $Q$ and of diameter $\frac{9}{8}\cdot(\mbox{diameter of }Q)$. We consider a set of cubes $\{Q_1, Q_2, Q_3\} \subset \mathcal{G}$ where $Q^*_1\cup Q^*_2\cup Q_3^*$ is connected and such that $Q^*_1 \cap Q^*_3=\emptyset$.  Similar to the previous lemma,  we will call $\{Q^*_1, Q_2^*, Q_3^*\}$ a {\bf \em chain} and call $Q^*_2$ the {\bf \em bridge}.  We have a similar claim.

\medskip
{\em Claim:}  For any chain $\{Q_1^*,Q_2^*, Q_3^*\}$, the bridge, $Q_2^*$, is strictly bigger than at least one of either $Q_1^*$ or $Q_3^*$.  In other words
\begin{align*}
\mathscr{H}^d(Q_2^*)> \mbox{min}(\mathscr{H}^d(Q_1^*), \mathscr{H}^d(Q_3^*)).
\end{align*}
Let us assume $2^{dk}=\mathscr{H}^d(Q_1)\leq \mathscr{H}^d(Q_3)=2^{d\ell}$. Define $Q^i_k$ for $k=1,2,3$ and $i=1,...,d$, as the projection of $Q_k$ onto the $i$th axis. We will rely heavily on the relationship between the properties of the dyadic cubes and the properties of their projections.  We first note that for cubes $Q$ and $H$
\begin{align*}
Q\cap H= \prod_{i=1}^d Q^i \cap H^i.
\end{align*}
Thus, $Q^*_1 \cap Q^*_3 = \emptyset$ implies that there is at least one index $p \in \{1,...,d\}$ such that $Q^{*p}_1 \cap Q^{*p}_3=\emptyset$.  We also note the commutativity of projection and dilation for cubes, i.e. $Q^{*p}=Q^{p*}$.  Thus,  we can assume $Q^{p*}_1 \cap Q^{p*}_3=\emptyset$.  Now of course $Q^p_1$ and $Q^p_3$ are dyadic intervals in $\mathbb{R}$, and $2^{k}=\mathscr{H}^1(Q^p_1)\leq \mathscr{H}^1(Q^p_3)=2^{\ell}$.  Since $Q^{p*}_1 \cap Q^{p*}_3=\emptyset$, $\mbox{dist}(Q^p_1, Q^p_3)= n\cdot 2^k$ for $n \in \mathbb{N}$.  By the argument from the previous lemma, the gap between $Q^{p*}_1$ and $Q^{p*}_2$ is strictly greater than $\frac{9}{32}2^k=\frac{9}{8}2^{k-2}$.  In order for $Q^*_1 \cup Q^*_2 \cup Q^*_3$ to be connected, $(Q^*_1 \cup Q^*_2 \cup Q^*_3)^i=Q^{*i}_1 \cup Q^{*i}_2 \cup Q^{*i}_3$ must be connected for each $i \in \{1,...,d\}$.  Therefore, 
\begin{align*}
\mathscr{H}^1(Q^{p*}_2) \geq \frac{1}{2}\mathscr{H}^1(Q_1^{p*})=\frac{9}{8}2^{k-1}.
\end{align*}

Thus, we have shown that $\mathscr{H}^d(Q^*_2) \geq \frac{1}{2^d}\mathscr{H}^d(Q_1^*)$ and, analogous to what we have done in the previous lemma, we would like to show that $\mathscr{H}^d(Q^*_2) \neq \mathscr{H}^d(Q_1^*)$ and $\mathscr{H}^d(Q^*_2) \neq \frac{1}{2^d}\mathscr{H}^d(Q_1^*)$. So, similar to Lemma \ref{lemma}, we will show that if $\mathscr{H}^d(Q_2^*)$ is equal to $\mathscr{H}^d(Q_1^*)$ or $\frac{1}{2^d}\mathscr{H}^d(Q_1^*)$  then  $Q_2^* \cap Q_1^*= \emptyset$.  However, we must be careful with how we approach this particular part of the argument.  We will not be able to show that $Q_2^{*p} \cap Q_1^{*p}= \emptyset$ since on the $p$th axis $Q_1^p \cap Q_2^p$ is not necessarily empty.  However, it is clear that there is an index $j \in \{1,...,d\}$ such that $Q_1^j$ and $Q_2^j$ are not adjacent (i.e. they do not share an endpoint).  Given this, we can let $m=k-i$, for $i\in\{0, 1\}$, then $Q^j_2$ and $Q^j_1$ being non adjacent implies $\mbox{dist}(Q^j_2,Q^j_1)\geq 2^m$. Then
\begin{align*}
\mbox{dist}(Q^{j*}_1,Q_2^{j*}) &\geq 2^m-\frac{1}{16}(2^k+2^m)\\
&=15\cdot 2^{m-4}-2^{k-4}\\
&\geq 15\cdot 2^{k-5}-2^{k-4}\\
&=13\cdot 2^{k-5}>0
\end{align*}
Thus $\mathscr{H}^d(Q^{*}_2)>\mathscr{H}^d(Q^{*}_1)$.  
This implies that for any chain $\{Q^*_1,Q^*_2,Q^*_3\} \subset \mathcal{G}^*$,
\begin{align*}
\mathscr{H}^d(Q^*_2) > \mbox{min}(\mathscr{H}^d(Q^*_1), \mathscr{H}^d(Q^*_3)).
\end{align*}
Let $\mathcal{G}^*=\{H_0, H_1, ..., H_n\}$ where the dilated cubes are listed in decreasing order by size. If each $H_i \in \mathcal{G}^*$ is contained in $H_0$ then we are done and $\cup_i H_i \subset 4 \cdot H_0$.  If not, let $\{H_{k_{\ell}}\}$ be the set of cubes such that $H_{k_{\ell}} \cap H_0 \neq \emptyset$ and $H_{k_{\ell}} \not\subset H_0$.  Then, by assumption, $\mathscr{H}^d(H_{k_{\ell}}) \leq \mathscr{H}^d(H_0)$ and as before we cannot have $H_{k_{\ell}} \cap H_0 \neq \emptyset$ and $\mathscr{H}^d(H_{k_{\ell}}) = \mathscr{H}^d(H_0)$.  Thus 
$$\mathscr{H}^d(H_{k_{\ell}}) < \mathscr{H}^d(H_0) \hspace{.5cm} (*)$$ 
for each $k_{\ell}$.  Hence for each $j$
\begin{align*}
\{ x_j: a_{0,j}\leq x_j \leq b_{0,j} \} \cap \{ x_j: a_{k_{\ell},j}\leq x_j \leq b_{k_{\ell},j} \} \neq \emptyset
\end{align*}
which is clear from assumption.  Then by $(*)$,
\begin{align*}
\mathscr{H}^1(\{ x_j: a_{0,j}\leq x_j \leq b_{0,j} \}) \geq 2\mathscr{H}^1(\{ x_j: a_{k_{\ell},j}\leq x_j \leq b_{k_{\ell},j} \})
\end{align*}
By the argument from Lemma \ref{lemma}
\begin{align*}
4 \cdot \{ x_j: a_{0,j}\leq x_j \leq b_{0,j} \} \supset 4\cdot \{ x_j: a_{k_{\ell},j}\leq x_j \leq b_{k_{\ell},j} \}
\end{align*}
Whence $4 \cdot H_0 \supset 4 \cdot H_{k_{\ell}}$ for each $k_{\ell}$ which implies
\begin{align*}
\bigcup_{\ell}4 \cdot H_{k_{\ell}} \subset 4 \cdot H_0.
\end{align*}
Now if each $H_i \in \mathcal{G}^*$ is contained in $H_0 \cup \bigcup_{\ell} H_{k_{\ell}}$ then we are done. If not, let $\{H_{m_{r}}\}$ be the set of cubes that satisfy the following:
 \begin{itemize}
 \item $H_{m_{r}} \cap H_{k_{\ell}} \neq \emptyset$ for at least one $\ell$
\item  $H_{m_r} \not\in \{H_{k_{\ell}}\}$
\item  $H_{m_r} \not\subset H_0$.
\end{itemize}
 Recall also have $H_{k_{\ell}} \not\subset H_0$ for all $\ell$.  For any $H_i$ such that $H_0 \cap H_i \neq \emptyset$ either $H_i \subset H_0$ or $H_i=H_{k_{\ell}}$ for some $\ell$.  Therefore, for any $m_r$,  $H_{m_r} \cap H_0 = \emptyset$.  Furthermore, for any $\ell$ and $r$ such that $H_{m_r} \cap H_{k_{\ell}} \neq \emptyset$, $\{H_0, H_{k_{\ell}}, H_{m_r}\}$ forms a chain with $H_{k_{\ell}}$ as the bridge.  Thus, $\mathscr{H}^d(H_0)> \mathscr{H}^d(H_{k_{\ell}})$ and $\mathscr{H}^d(H_{k_{\ell}})> \min(\mathscr{H}^d(H_{0}),\mathscr{H}^d(H_{m_{r}}))$ implies that $\mathscr{H}^d(H_{k_{\ell}})> \mathscr{H}^d(H_{m_{r}})$.  By the same argument that shows us that $4 \cdot H_0 \supset 4 \cdot H_{k_{\ell}}$, we have $4 \cdot H_{k_{\ell}} \supset 4 \cdot H_{m_r}$.  

Again if each $H_i \in \mathcal{G}^*$ is contained in $H_0 \cup \bigcup_{\ell} H_{k_{\ell}} \cup \bigcup_{r} H_{m_{r}}$, then we are done.  Otherwise the process continues and, at each step, we get the same relation as we have between the $H_{k_{\ell}}$, $H_{m_r}$ and $H_0$: 
\begin{align*}
\bigcup_{r} 4 \cdot H_{m_{r}} \subset \bigcup_{\ell} 4 \cdot H_{k_{\ell}}  \subset 4 \cdot H_0 \Rightarrow \bigcup_{j} H_{m_{r}} \subset 4 \cdot H_0
\end{align*}
Since there are only finitely many $H_i$, this process has to exhaust the collection $\{H_i\}$.
\end{proof}
\section{Fourier Series}
We turn our attention to $S_Nf=D_N*f$, where $S_Nf$ in the $N$-th partial sum of the trigonometric Fourier series of $f$ and $D_N$ is the Dirichlet Kernel.  We note that with $S_Nf$ replacing $f$ in Prop. \ref{1st reduction} we are considering the integral
\begin{align*}
\int_{\mathbb{T} \setminus E} |S_Nf(\theta)|^2 \,d\theta
\end{align*}
However, we have no hope of bounding this by $\lambda \|f\|^2_1$ or, equivalently, showing that for each $f \in L^1(\mathbb{T})$ ($\|f\|_1=1$) there exists an $E \subset \mathbb{T}$ such that $|E| \lesssim 1/\lambda$ and
\begin{align*}
 \int_{\mathbb{T} \setminus E} |S_Nf(\theta)|^2 \,d\theta \leq C\lambda
\end{align*}
 for all $N\geq 1$.  Of course, if we were to allow $E$ to depend on $N$ and replace $\lambda$ with $\lambda^2$, this could be accomplished, since the operator $S_N$ is weak $L^1$ bounded.  In our case, the bounds are not possible due to the example of Kolmogoroff's sequence of resonance measures (\cite{schlag13} Ch. 6), $\mu_n$ , such that $\|\mu_n\|=1$ and
\begin{align*}
\limsup_{N \rightarrow \infty} \,(\log n)^{-1}|S_N\mu_n(x)|>0
\end{align*}
for almost every $x \in \mathbb{T}$.  However, if we take the average of the first $N$ of these integrals, we obtain the following result:
\begin{theorem} \label{first thing}
Let $\lambda>0$.  Then for any $f \in L^1(\mathbb{T})$, there exists $E \subset \mathbb{T}$, with $|E| \leq \frac{1}{\lambda}$ such that 
\begin{align} \label{first}
\sup_{N \geq 1} \frac{1}{N} \sum_{n=1}^N \int_{\mathbb{T} \setminus E} |S_{n}f(\theta)|^2\,d\theta \lesssim \lambda \|f\|_1^2  
\end{align}
where $S_nf(x)= \mathcal{F}^{-1}(\hat{f}\chi_{[-n,n]})(x)$.
\end{theorem}
\begin{proof}
Assume $\|f\|_1$=1 and $f\geq 0$. We first fix $N \in \mathbb{N}$, $\lambda >0$, and consider the Dirichlet kernel of $S_n$, $D_n(x)=\sum_{m=-n}^{n} e(mx)$, where $e(mx):= e^{2\pi i m x}$.  We note that
\begin{align*}
D_n(x)= \sum_{m=-n}^{n} e(mx)=\frac{e((n+1)x)-e(-nx)}{e(x)-1}= K_{n,1}(x) +K_{n,2}(x)
\end{align*}
where $K_{n,1}(x):= D_n(x)\chi_{[|x|\leq 1/N]}$ and $K_{n,2}(x):= D_n(x)\chi_{[|x|> 1/N]}$.
Then for $K_{n,2}$, by Plancherel
\begin{align*}
\sum_{n=1}^N|(K_{n,2}*f)(\theta)|^2 &\lesssim \sum_{m \in \mathbb{Z}} |\widehat{F}_{\theta}(m)|^2= \int_{|x|>1/N}\frac{|f(\theta-x)|^2}{|e(x)-1|^2} \, dx \\    
& \lesssim \int_{\mathbb{T}}\min(N^2, |x|^{-2})|f(\theta-x)|^2\,dx 
\end{align*}
where $F_{\theta}(x)=\chi_{[|x|> 1/N]}\frac{f(\theta-x)}{e(x)-1}$.  For $K_{n,1}$, we note that $|K_{n,1}(x)|\lesssim N$ for all $1 \leq n \leq N$ and $x \in \mathbb{T}$.  Thus, by Jensen,
\begin{align*} 
\sum_{n=1}^N |(K_{n,1}*f)(\theta)|^2 &\lesssim \sum_{n=1}^N \int_{\mathbb{T}} \chi_{[|x|\leq 1/N]} N |f(\theta-x)|^2\,dx \\ 
&\lesssim \int_{\mathbb{T}}\min(N^2, |x|^{-2})|f(\theta-x)|^2\,dx 
\end{align*}
Let $K_N(x):= \frac{1}{N}\min(N^2,|x|^{-2})$, and we note that $\|K_N\|_1$ is independent of $N$, then
\begin{align} \label{main split}
\frac{1}{N} \sum_{n=1}^N \int_{\mathbb{T} \setminus E} |S_{n}f(\theta)|^2\,d\theta \leq  2\int_{\mathbb{T} \setminus E} \int_{\mathbb{T}}K_N(x)|f(\theta-x)|^2\,dx d\theta 
\end{align}
for any set $E \subset \mathbb{T}$.  We now perform a Calder\'{o}n-Zygmund decomposition at height $\lambda$.  Let $\mathcal{B}$ be the set of "bad" intervals given by the decomposition.  Let $f=g+b$ where $g$ is supported outside of the union of the intervals in $\mathcal{B}$, $|g|\leq \lambda$,  $b=\sum_{I \in \mathcal{B}}\chi_I f$ and 
\begin{align*}
\lambda < \frac{\int_I |f|}{|I|} \leq 2 \lambda \\
\mbox{and } \left|\bigcup_{I \in \mathcal{B}} I\right| \leq \frac{\|f\|_1}{\lambda}=\frac{1}{\lambda}.
\end{align*} 
Here we observe that we can regularize $f$ without loss of generality.  We consider $f^{(N)}=V_N*f$, $g^{(N)}=V_N*g$ and $b^{(N)}=V_N*b$ where $V_N=\frac{1}{N}\sum_{m=N}^{2N-1}S_m$ is the de la Vall\'{e}e Poussin kernel.  Note that if we replace $f$ in \eqref{first} with $f^{(N)}$, by Young's inequality
\begin{align*}
\frac{1}{N} \sum_{n=1}^N \int_{\mathbb{T} \setminus E} |S_{n}f^{(N)}(\theta)|^2\,d\theta &\lesssim\lambda \|f^{(N)}\|_1^2\\ 
 &\lesssim \lambda \|V_N\|_1^2 \|f\|_1^2 \\
&\lesssim \lambda \|f\|_1^2.
\end{align*}
This regularity will be used almost exclusively for Bernstein's inequality which will give us $\|f^{(N)}\|_{\infty} \lesssim N\|f\|_1$.  This fortunately also preserves $|g^{(N)}(x)| \lesssim\lambda $ by Young's inequality. We then have
\begin{align*} 
&\int_{\mathbb{T} \setminus E} \int_{\mathbb{T}}K_N(x)|f^{(N)}(\theta-x)|^2\,dx d\theta \\
&\lesssim \int_{\mathbb{T} \setminus E} \int_{\mathbb{T}}K_N(x)|g^{(N)}(\theta-x)|^2\,dx d\theta+\int_{\mathbb{T} \setminus E} \int_{\mathbb{T}}K_N(x)|b^{(N)}(\theta-x)|^2\,dx d\theta.
\end{align*}
Using $|g^{(N)}(x)| \lesssim \lambda$ with Young's inequality we obtain
\begin{align*} 
&\int_{\mathbb{T} \setminus E} \int_{\mathbb{T}}K_N(x)|g^{(N)}(\theta-x)|^2\,dx d\theta \\
&\lesssim \lambda \int_{\mathbb{T} \setminus E} (K_N*|g^{(N)}|)(\theta) \,d\theta \leq \lambda \|K_N\|_1\|g^{(N)}\|_1 \\
&\lesssim \lambda. 
\end{align*}
Inequality \eqref{main split} shows that this is all we need for $g$. Now since $b^{(N)}=V_N*b$, it is possible that $b^{(N)}$ is supported on all $\mathbb{T}$. Let $\mathcal{B}=\mathcal{B}_1 \cup \mathcal{B}_2$, where $\mathcal{B}_1$ is the set of all bad intervals of length greater than $1/N$ and $\mathcal{B}_2$ are the intervals in $\mathcal{B}$ of length less than or equal to $1/N$. Similarly, let $b=b_1+b_2$, where $b_1=\sum_{I \in \mathcal{B}_1}f_I$ and $b_2=\sum_{I \in \mathcal{B}_2} f_I$. Here $f_I:= f \chi_I$. Now we can choose $E:=\cup_{I \in \mathcal{B}} \; c \cdot I$. Where $c\cdot I$ is the interval with the same center as $I$, but of length $c|I|$.  $c$ is a constant whose value will be decided later which will not depend on $N$. Furthermore, let $I^*:=\frac{9}{8}\cdot I$ and $V_N*b_1=b^*+\tilde{b}$, where 
\begin{align*}
b^*=\sum_{I \in \mathcal{B}_1}\chi_{I^*}(V_N*f_I) \hspace{.4cm} \mbox{ and } \hspace{.4cm}\tilde{b}=\sum_{I \in \mathcal{B}_1}\chi_{\mathbb{T}\setminus  I^*}(V_N*f_I).
\end{align*}  
Then we first consider $b^*$
\begin{align*} 
&\int_{\mathbb{T} \setminus E} \int_{\mathbb{T}}K_N(x)|b^*(\theta-x)|^2\,dx d\theta \\
 &=\int_{\mathbb{T} \setminus E} \int_{\mathbb{T}}K_N(x)\left|\sum_{I \in \mathcal{B}_1 }\chi_{ I^*}(\theta-x) f^{(N)}_I(\theta-x)\right|^2\,dx d\theta. \end{align*}
We would like to use the assumption that $\mathcal{B}_1$ is a collection of disjoint intervals, but $\{ I^*\}_{I \in \mathcal{B}_1}$ is not a collection of disjoint intervals. We alternatively label the connected components of $\cup_{I \in \mathcal{B}_1} I^*$ as $\mathcal{C}_i$. Then
\begin{align*}
&\int_{\mathbb{T} \setminus E} \int_{\mathbb{T}}K_N(x)\left|\sum_{I \in \mathcal{B}_1 }\chi_{ I^*}(\theta-x) f^{(N)}_I(\theta-x)\right|^2\,dx d\theta \\
&=\sum_i \int_{\mathbb{T} \setminus E} \int_{\mathbb{T}}K_N(x)\left|\sum_{I^* \subset \mathcal{C}_i }\chi_{ I^*}(\theta-x) f^{(N)}_I(\theta-x)\right|^2\,dx d\theta \\
&\leq \sum_i \int_{\mathbb{T} \setminus E} \int_{\mathbb{T}}K_N(x)\left(\sum_{I^* \subset \mathcal{C}_i }\chi_{ I^*}(\theta-x) \left|f^{(N)}_I(\theta-x)\right|\right)^2\,dx d\theta.
\end{align*}
Note that in the second and third lines $I \in \mathcal{B}_1$.  For any $I^* \subset \mathcal{C}_i$, $\chi_{I^*} \leq \chi_{\mathcal{C}_i}$, and thus
\begin{align*}
 &\sum_i \int_{\mathbb{T} \setminus E} \int_{\mathbb{T}}K_N(x)\left(\sum_{I^* \subset \mathcal{C}_i }\chi_{ I^*}(\theta-x) \left|f^{(N)}_I(\theta-x)\right|\right)^2\,dx d\theta \\
&\lesssim  \sum_i \int_{\mathbb{T} \setminus E} \int_{\mathbb{T}}K_N(x)\left(\sum_{I^* \subset \mathcal{C}_i }\chi_{\mathcal{C}_i}(\theta-x) \left|f^{(N)}_I(\theta-x)\right|\right)^2\,dx d\theta.
\end{align*}
 Now before we can move on, we need to impose a restriction on the geometry of our collection $\mathcal{B}_1$.  It would be helpful if for any pair of intervals in $\mathcal{B}_1$, the two intervals are not adjacent (i.e.\ they do not share an endpoint).  This is an easy restriction to impose by simply splitting $\mathcal{B}_1$ into at most 3 subcollections.  We lose a factor of 3 in the inequality, but this allows us to invoke Lemma \ref{lemma}.  For each $\mathcal{C}_i$, let $J_i$ be the largest interval such that $J_i \in \mathcal{B}_1$ and $J_i^* \subset \mathcal{C}_i$. Lemma \ref{lemma} implies that $\mathcal{C}_i \subset 4 \cdot J^*_i=\frac{9}{2} \cdot J_i$ and therefore $\chi_{\mathcal{C}_i} \leq \chi_{4 \cdot J^*_i}$ giving us the following: for any $i$,
\begin{align*}
&\int_{\mathbb{T} \setminus E} \int_{\mathbb{T}}K_N(x)\left(\sum_{I^* \subset \mathcal{C}_i } \chi_{\mathcal{C}_i}(\theta-x)\left|f^{(N)}_I(\theta-x)\right|\right)^2\,dx d\theta \\
&\leq \int_{\mathbb{T} \setminus E} \int_{\mathbb{T}}K_N(x)\left(\sum_{I^* \subset \mathcal{C}_i }\chi_{4\cdot J^*_i}(\theta-x) \left|f^{(N)}_I(\theta-x)\right|\right)^2\,dx d\theta\\
&= \int_{\mathbb{T} \setminus E} \int_{\mathbb{T}}K_N(x)\left(\chi_{4\cdot J^*_i}(\theta-x) \sum_{I^* \subset \mathcal{C}_i }\left|f^{(N)}_I(\theta-x)\right|\right)^2\,dx d\theta.
\end{align*}
Henceforth we will take $c$ to be 5 for $E=\bigcup_{I \in \mathcal{B}} c \cdot I$ and replace $\mathbb{T} \setminus E$ with $\mathbb{T} \setminus c \cdot J_i$ in the integral:
\begin{align*}
& \int_{\mathbb{T} \setminus E} \int_{\mathbb{T}}K_N(x)\left(\chi_{4\cdot J^*_i}(\theta-x)\sum_{ I^* \subset \mathcal{C}_i } \left|f^{(N)}_I(\theta-x)\right|\right)^2\,dx d\theta \\
&\leq \frac{1}{N} \int_{\mathbb{T} \setminus c\cdot J_i} \int_{\mathbb{T}}\min(N^2, |x|^{-2})\left(\chi_{4\cdot J^*_i}(\theta-x)\sum_{I^* \subset \mathcal{C}_i } \left|f^{(N)}_I(\theta-x)\right|\right)^2\,dx d\theta 
\end{align*}
$\theta-x \in 4 \cdot J^*_i=\frac{9}{2} \cdot J_i$ and $\theta \in \mathbb{T}\setminus 5 \cdot J_i$, so $|x| \geq \frac{1}{4}|J_i| > \frac{1}{4N}$ and thus
\begin{align*}
&\frac{1}{N} \int_{\mathbb{T} \setminus c\cdot J_i} \int_{\mathbb{T}}\min(N^2, |x|^{-2})\left(\chi_{4\cdot J^*_i}(\theta-x)\sum_{I^* \subset \mathcal{C}_i } \left|f^{(N)}_I(\theta-x)\right|\right)^2\,dx d\theta \\
&\lesssim\frac{1}{N}\int_{\mathbb{T} \setminus c\cdot J_i} \int_{\mathbb{T}}|x|^{-2}\left(\chi_{4\cdot J^*_i}(\theta-x)\sum_{ I^* \subset \mathcal{C}_i } \left|f^{(N)}_I(\theta-x)\right|\right)^2\,dx d\theta.
\end{align*}
Now Bernstein (or Young) gives the inequality $\sum_{I^* \in \mathcal{C}_i } \left|f^{(N)}_I\right|\leq N\sum_{I^* \subset \mathcal{C}_i } \|f_I\|_1$, so
\begin{align*}
&\frac{1}{N}\int_{\mathbb{T} \setminus c\cdot J_i} \int_{\mathbb{T}}|x|^{-2}\left(\chi_{4\cdot J^*_i}(\theta-x)\sum_{ I^* \subset \mathcal{C}_i } \left|f^{(N)}_I(\theta-x)\right|\right)^2\,dx d\theta \\
& \lesssim \left(\sum_{I^* \subset \mathcal{C}_i }\|f_I\|_1\right)\int_{\mathbb{T} \setminus c\cdot J_i} \int_{\mathbb{T}} |x|^{-2}\chi_{4\cdot J^*_i}(\theta-x)\left(\sum_{I^* \subset \mathcal{C}_i } \left|f^{(N)}_I(\theta-x)\right|\right)\,dx d\theta \\
 &\lesssim\left(\sum_{I^* \subset \mathcal{C}_i }\|f_I\|_1\right)\int_{\mathbb{T}}\chi_{4\cdot J^*_i}(x)\sum_{I^* \subset \mathcal{C}_i } |f^{(N)}_I(x)|\int_{\mathbb{T} \setminus c\cdot J_i}  |\theta-x|^{-2}\, d\theta dx 
\end{align*}
By H\"{o}lder
\begin{align*} 
&\left(\sum_{I^* \subset \mathcal{C}_i }\|f_I\|_1\right)\int_{\mathbb{T}}\chi_{4\cdot J^*_i}(x)\sum_{I^* \subset \mathcal{C}_i } |f^{(N)}_I(x)|\int_{\mathbb{T} \setminus c\cdot J_i}  |\theta-x|^{-2}\, d\theta dx \\
&\lesssim \left(\sum_{I^* \subset \mathcal{C}_i }\|f_I\|_1\right)\left(\sum_{I^* \subset \mathcal{C}_i }\|f^{(N)}_I\|_1\right) \int_{\mathbb{T}} ||J_i| + \theta|^{-2} \,d\theta \\ 
&\lesssim \left(\sum_{I^* \subset \mathcal{C}_i }\|f_I\|_1\right)^2 |J_i|^{-1} 
\end{align*}
Now from the properties of the Calder\'{o}n-Zygmund decomposition
\begin{align*}
\left(\sum_{I^* \subset \mathcal{C}_i }\|f_I\|_1\right) |J_i|^{-1} &=\sum_{I^* \subset \mathcal{C}_i }\frac{\|f_I\|_1}{|J_i|} \\
&\lesssim \sum_{I^* \subset \mathcal{C}_i }\frac{\lambda|I|}{|J_i|} = \lambda \frac{\sum_{I^* \subset \mathcal{C}_i}|I|}{|J_i|} \\
&\lesssim \lambda\frac{\frac{9}{2}|J_i|}{|J_i|} \hspace{.5cm} \mbox{By Lemma } \ref{lemma}\\
&\lesssim\lambda
\end{align*}
Summing over the connected components yields
\begin{align*}
&\sum_i \int_{\mathbb{T} \setminus E} \int_{\mathbb{T}}K_N(x)\left|\sum_{I^* \subset \mathcal{C}_i }\chi_{ I^*}(\theta-x) f^{(N)}_I(\theta-x)\right|^2\,dx d\theta \\
&\lesssim \sum_i \left(\sum_{I^* \subset \mathcal{C}_i }\|f_I\|_1\right)^2 |J_i|^{-1} \lesssim \sum_i \lambda \left(\sum_{I^* \subset \mathcal{C}_i }\|f_I\|_1\right) =\lambda \sum_{I \in \mathcal{B}_1} \|f_I\|_1 \\
&\lesssim \lambda
\end{align*}
as desired. Now, for $\tilde{b}$, we have
\begin{align*} 
\int_{\mathbb{T} \setminus E} \int_{\mathbb{T}}K_N(x)|\tilde{b}(\theta-x)|^2\,dx d\theta &=\int_{\mathbb{T} \setminus E} (K_N*|\tilde{b}|^2)(\theta)\, d\theta\\
&\lesssim \|K_N*|\tilde{b}|^2\|_1 \leq \|K_N\|_1\|\tilde{b}\|_2^2 \hspace{.4cm}\mbox{By Young's inequality}\\
 &\lesssim \int_{\mathbb{T}}\left|\tilde{b}(x) \right|^2\,dx= \int_{\mathbb{T}}\left|\sum_{I \in \mathcal{B}_1 }\chi_{\mathbb{T}\setminus I^*}(x) f^{(N)}_I(x)\right|^2\,dx \\
&\lesssim \int_{\mathbb{T}} \sum_{I \in \mathcal{B}_1} \chi_{\mathbb{T}\setminus I^*}(x)\left|f_I^{(N)}(x)\right|\left|\sum_{I \in \mathcal{B}_1} \chi_{\mathbb{T}\setminus I^*}f^{(N)}_I(x)\right|dx .
\end{align*}
 
We are done with $b^*$ and we include the following claim only for completeness. Indeed, we have handled the case when $I^*=2\cdot I$ in Proposition \ref{2nd reduction}, and the argument is essentially the same absent a change in constants.

\medskip
 {\bf Claim: }$\|\tilde{b}\|_{\infty}=\|\sum_{I}\chi_{\mathbb{T}\setminus I^*}f^{(N)}_I\|_{\infty} \lesssim \lambda$
\medskip

 We consider the two possible cases for any $x \in \mathbb{T}$: either (1) $x \in \mathbb{T} \setminus \cup I^*$ or (2) $x \in I^*$ for at least one $I \in \mathcal{B}_1.$ In the first case, for any $x \in \mathbb{T} \setminus \cup I^*$, $\chi_{\mathbb{T}\setminus I^*}(x)=1$ so
\begin{align*}
\left|\tilde{b}(x)\right|&\lesssim \sum_{I \in \mathcal{B}_1} \chi_{\mathbb{T}\setminus I^*}(x)|f^{(N)}_I(x)|=\sum_{I \in \mathcal{B}_1} \chi_{\mathbb{T}\setminus I^*}(x)|(V_N*f_I)(x)|=\sum_{I \in \mathcal{B}_1}|(V_N*f_I)(x)| \\
&\lesssim \sum_{I \in \mathcal{B}_1}(|V_N|*|f_I|)(x)
\end{align*}
In the second case, $x \in H_k^*$ for some subcollection of $\mathcal{B}_1$, $\{H_k\}$. Then  $\chi_{\mathbb{T}\setminus I^*}(x)=1$ for $I^* \not\in \{H_k\}$ and $\chi_{\mathbb{T}\setminus H_k^*}(x)=0$, so 
\begin{align*}
\tilde{b}(x)&=\sum_{I \in \mathcal{B}_1} \chi_{\mathbb{T}\setminus I^*}(x)|f^{(N)}_I(x)|=\sum_{I \in \mathcal{B}_1 \atop I \not\in \{H_k\}} \chi_{\mathbb{T}\setminus I^*}(x)|f^{(N)}_I(x)| \\
&\lesssim \sum_{I \in \mathcal{B}_1 \atop I \not\in \{H_k\}} (|V_N|* |f_I|)(x).
\end{align*}
In this case, for every $|V_N|*|f_I|$ in the sum, $x \in \mathbb{T} \setminus I^*$. So in both cases, we are taking a sum of $(|V_N|*|f_I|)(x)$ where $x \not\in \cup I^*$ and the union is taken over the same intervals as the sum.  Therefore, it suffices to assume that we are in the first case and $x \in \mathbb{T} \setminus \cup I^*$, where the union is taken over all $I \in \mathcal{B}_1$.  So we fix some $x  \in \mathbb{T} \setminus \cup I^*$.  First in order to bound each $|V_N|* |f_I|$, we recall that $|V_N(y)| \lesssim \frac{1}{N} \mbox{min}(N^2, |y|^{-2})$, then
\begin{align*}
&|\tilde{b}(x)| \lesssim \sum_{I \in \mathcal{B}_1}(|V_N|*|f_I|)(x) = \sum_{I \in \mathcal{B}_1} \int_{\mathbb{T}} |V_N(x-y)| |f_I(y)|\,dy \\
&\lesssim \int_{\mathbb{T}} \frac{1}{N}\mbox{min}(N^2,|x-y|^{-2})\sum_{I \in \mathcal{B}_1} |f_I(y)|\,dy
\end{align*}
The $|f_I(y)|$ gives us that the product is nonzero for $y \in I$, and by assumption $x \in \mathbb{T} \setminus I^*$.  Therefore, $|x-y|> \frac{1}{16}|I|>\frac{1}{16N}$(for every $I \in \mathcal{B}_1$) which implies $|V_N(x-y)| \lesssim \frac{1}{N}|x-y|^{-2}$ and 
\begin{align*}
&\int_{\mathbb{T}} \frac{1}{N}\mbox{min}(N^2,|x-y|^{-2})\sum_{I \in \mathcal{B}_1} |f_I(y)|\,dy \\
&\lesssim \frac{1}{N}\sum_{I \in \mathcal{B}_1} \sup_{y \in I}|x_0-y|^{-2} \|f_I\|_1 & \mbox{By H\"{o}lder}
\end{align*}
for any fixed $x \in \mathbb{T}\setminus \cup I^*$.  For each $I$, $|x-y|$ with $y \in I$ is at least $\frac{1}{16}|I|$.  Of course, $|x-y|$ will be greater than the distance between $x$ and $I$ for any $y \in I$. Then $|x-y|\gtrsim \max(|I|, \mbox{dist}(x,I)) $ which implies $|x-y| \gtrsim |I|+\mbox{dist}(x,I) $.  Thus by ordering $\mathcal{B}_1=\{I_j\}$ by proximity to $x$ (considering only those intervals to the right of $x$ without loss of generality) we have
\begin{align*}
&\frac{1}{N}\sum_{I \in \mathcal{B}_1} \sup_{y \in I}|x-y|^{-2} \|f_I\|_1 \\
&\lesssim \frac{1}{N} \sum_{j =1}^{|\mathcal{B}_1|} (|I_j|+\mbox{dist}(x,I_j))^{-2}\|f_{I_j}\|_1 \\
&\lesssim \frac{\lambda}{N} \sum_{j=1}^{|\mathcal{B}_1|} \frac{|I_j|}{(|I_j|+\mbox{dist}(x,I_j))^{2}} & \mbox{By the C-Z decomposition.}
\end{align*}
Let $\phi_x(y):=\min(N^{2}, |x-y|^{-2})$.  Then, for all $y \in I_j$,
\begin{align*}
\frac{1}{(|I_j|+\mbox{dist}(x,I_j))^{2}} &\lesssim \frac{1}{|x-y|^{2}}   \lesssim \phi_x(y) \\
\Rightarrow  \frac{1}{(|I_j|+\mbox{dist}(x,I_j))^{2}} &\lesssim \inf_{y \in I_j}\phi_x(y).
\end{align*}
Recall that the $I_j$ are pairwise disjoint. Therefore, the sum $\sum_{j=1}^{|\mathcal{B}_1|} \frac{|I_j|}{(|I_j|+\mbox{dist}(x,I_j))^{2}}$ is bounded by a lower Riemann sum of $\phi_x(y)$.  Thus,
\begin{align*}
\frac{1}{N} \sum_{j=1}^{|\mathcal{B}_1|} \lambda\frac{|I_j|}{(|I_j|+\mbox{dist}(x,I_j))^{2}} &\leq \frac{\lambda}{N} \|\phi_x\|_1 \\
&\lesssim  \frac{\lambda}{N} N =\lambda.
\end{align*}
In conclusion, the claim holds and $\|\tilde{b}\|_{\infty} \lesssim \lambda$, which implies
\begin{align*}
\int_{\mathbb{T}\setminus E} (K_N*|\tilde{b}|^2)(x) \,dx &\lesssim \|K_N\|_1 \int_{\mathbb{T}} \sum_{I \in \mathcal{B}_1} \chi_{\mathbb{T}\setminus I^*}(x)|f_I^{(N)}(x)|\left|\sum_{I \in \mathcal{B}_1} \chi_{\mathbb{T}\setminus I^*}f^{(N)}_I(x)\right|\,dx  \\
&\lesssim \|K_N\|_1 \lambda \int_{\mathbb{T}} \sum_{I \in \mathcal{B}_1} \chi_{\mathbb{T}\setminus I^*}(x)|f_I^{(N)}(x)|\,dx \hspace{.5cm} (\mbox{from } \|\tilde{b}\|_{\infty} \lesssim \lambda)\\
&\lesssim \lambda   \sum_{I \in \mathcal{B}_1}\int_{\mathbb{T}}|f_I^{(N)}(x)|\,dx \lesssim \lambda \sum_{I \in \mathcal{B}_1} \|f_I\|_1 \\
&\lesssim \lambda
\end{align*}
This gives us
\begin{align*}
&\int_{\mathbb{T} \setminus E} \int_{\mathbb{T}}K_N(x)|b_1^{(N)}(\theta-x)|^2\,dx d\theta \\
& \lesssim \int_{\mathbb{T} \setminus E} \int_{\mathbb{T}}K_N(x)|b^*(\theta-x)|^2\,dx d\theta +\int_{\mathbb{T} \setminus E} \int_{\mathbb{T}}K_N(x)|\tilde{b}(\theta-x)|^2\,dx d\theta \\
&\lesssim \lambda+\lambda \\
&\lesssim \lambda
\end{align*}

Again we include the following bound for $b^{(N)}_2$ for completeness.  Assume that $\frac{1}{N}=2^j$ for some $j \in \mathbb{Z}$ and partition $\mathbb{T}$ into $N$ intervals of length $\frac{1}{N}$.  We can do this while giving away a factor of two in the final bound.  Then since $I \in \mathcal{B}_2$ are dyadic, each $I$ is contained in a length $\frac{1}{N}$ interval.  Then let
\begin{align*}
b_2^{(N)}&=\sum_{I \in \mathcal{B}_2}f^{(N)}_I= \sum_{J} \sum_{I \in \mathcal{B}_2 \atop I \subset J}f_I^{(N)} 
\end{align*}
where the $J$ intervals come from the $1/N$ partition. Then, for the remainder of this proof, we let 
\begin{align*}
f_J:= \chi_J b_2= \sum_{I \in \mathcal{B}_2 \atop I \subset J}f_I.
\end{align*}
 We first note again that $|V_N(\theta)| \lesssim \frac{1}{N} \min(|\theta|^{-2}, N^2)$ and thus 
\begin{align*}
|f^{(N)}_J(y)| &\lesssim \int_{\mathbb{T}}\frac{1}{N}\min(|y-x|^{-2},N^2)|f_J(x)| \,dx \\
& \lesssim \frac{1}{N}\|f_J\|_1  \min(\sup_{x \in J}|y-x|^{-2},N^2)\\
& \lesssim \frac{1}{N}\|f_J\|_1 \frac{1}{\mbox{dist}(y,J)^2+\frac{1}{N^2}}
\end{align*}
Then
\begin{align*}
&\left| \sum_J f^{(N)}_J(y) \right|^2 \lesssim  \left|\sum_J f^{(N)}_J(y)\right|\left( \sum_J \left|f^{(N)}_J(y)\right| \right)\\
& \lesssim \left|b^{(N)}_2(y)\right| \sum_J \|f_J\|_1 \frac{1/N}{\mbox{dist}(y, J)^2+\frac{1}{N^2}} \\
&\lesssim \left|b^{(N)}_2(y)\right|(\max_J \|f_J\|_1) \sum_J  \frac{1/N}{\mbox{dist}(y, J)^2+\frac{1}{N^2}}
\end{align*}
We note again that by the Calder\'{o}n-Zygmund decomposition for any $J$,
\begin{align*}
&\|f_J\|_1=\frac{1}{N}\frac{\|f_J\|_1}{|J|} \leq \frac{1}{N}\frac{\sum_{I \in \mathcal{B}_2 \atop I \subset J} \|f_I\|_1}{|J|} \\
& \lesssim  \frac{1}{N}\frac{\lambda \sum_{I \in \mathcal{B}_2 \atop I \subset J} |I|}{|J|} \leq \frac{1}{N}\frac{\lambda |J|}{|J|}\\
&\lesssim \lambda|J| =\frac{\lambda}{N}
\end{align*}
Therefore,
\begin{align*}
& \left|b^{(N)}_2(y)\right|(\max_J \|f_J\|_1)\sum_J \frac{1/N}{\mbox{dist}(y, J)^2+\frac{1}{N^2}} \\
&\lesssim \lambda\left|b^{(N)}_2(y)\right| \sum_J  \frac{1/N^2}{\mbox{dist}(y, J)^2+\frac{1}{N^2}}
\end{align*}
If we fix $y$, then for each $J$, there is a nonnegative constant $C_y \leq \frac{1}{N}$ that doesn't depend on $J$ and a positive integer $m_J \in [1,N]$ unique to each $J$ such that $\mbox{dist}(y,J) = C_y+m_J\frac{1}{N} $
\begin{align*} 
&\sum_J  \frac{1/N^2}{\mbox{dist}(y, J)^2+\frac{1}{N^2}} \lesssim \sum_J  \frac{1/N^2}{(C_y+m_J\frac{1}{N})^2+\frac{1}{N^2}}\lesssim \sum_J  \frac{1/N^2}{m^2_J\frac{1}{N^2}+\frac{1}{N^2}} \\
&\lesssim \sum_{i=1}^N \frac{1}{i^2} \lesssim 1 
\end{align*}
This gives us the following inequality for all $y \in \mathbb{T}$:
\begin{align*}
\left| \sum_J f^{(N)}_J(y) \right|^2 \lesssim \lambda\left| \sum_J f^{(N)}_J(y) \right| =\lambda\left|b^{(N)}_2(y)\right|.
\end{align*}
Thus 
\begin{align*}
&\int_{\mathbb{T} \setminus E} \int_{\mathbb{T}} K_N(x) \left| \sum_J f^{(N)}_J(y-x) \right|^2 \,dx dy \\
&\lesssim \lambda \int_{\mathbb{T}} \int_{\mathbb{T}} K_N(x) \left|b^{(N)}_2(y-x)\right| \,dx dy= \lambda \|K_N * |b_2^{(N)}|\|_1 \\
&\lesssim \lambda  \|K_N\|_1  \|b_2^{(N)}\|_1 \hspace{.3cm} \mbox{By Young's Inequality} \\
&\lesssim \lambda \|b_2\|_1\lesssim \lambda\|f\|_1\lesssim \lambda \end{align*}
as desired.  The three main estimates we have obtained combine to give us
\begin{align*}
&\frac{1}{N} \sum_{n=1}^N \int_{\mathbb{T} \setminus E} |S_{n}f^{(N)}(\theta)|^2\,d\theta \\
&\lesssim \int_{\mathbb{T} \setminus E} \int_{\mathbb{T}}K_N(x)|f^{(N)}(\theta-x)|^2\,dx d\theta \\
&\lesssim \int_{\mathbb{T} \setminus E} \int_{\mathbb{T}}K_N(x)|g^{(N)}(\theta-x)|^2\,dx d\theta +\int_{\mathbb{T} \setminus E} \int_{\mathbb{T}}K_N(x)|b_1^{(N)}(\theta-x)|^2\,dx d\theta \\
& \hspace{2cm}+\int_{\mathbb{T} \setminus E} \int_{\mathbb{T}}K_N(x)|b_2^{(N)}(\theta-x)|^2\,dx d\theta \\
&\lesssim \lambda 
\end{align*}
Through scaling we get with $\|f\|_1$ and $\lambda$ we can assume $|E|\leq \frac{1}{\lambda}$ and 
\begin{align*}
&\frac{1}{N} \sum_{n=1}^N \int_{\mathbb{T} \setminus E} |S_{n}f^{(N)}(\theta)|^2\,d\theta \lesssim \lambda\|f\|_1^2.
\end{align*}
This is precisely the bound we sought out.
\end{proof}
Before we move on and extend the result to the real line, let us revisit a question posed before Lemma \ref{lemma}.  Specifically, we asked for which $s>1$ does Proposition \ref{2nd reduction} still hold when $Q_N(y)=\frac{1}{N^{s-1}} \min(N^s, |y|^{-s})$ replaces $B_N$.  The answer to this question is a direct corollary to the reduction of Theorem \ref{first thing} to the estimate
\begin{align*}
\int_{\mathbb{T} \setminus E} \int_{\mathbb{T}}K_N(x)|f^{(N)}(\theta-x)|^2\,dx d\theta=\int_{\mathbb{T} \setminus E} K_N* |f^{(N)}|^2 \lesssim \lambda.
\end{align*}
\begin{corollary} \label{decay corollary}
Let $N \in \mathbb{N}\setminus \{0\}$, $\lambda>0$, and $s\geq2$.  Then for any $f \in L^1(\mathbb{T})$ such that $\mbox{supp}(\hat{f}) \subset [-N,N]$, there exists $E \subset \mathbb{T}$, with $|E| \leq \frac{1}{\lambda}$ such that 
\begin{align*}
\int_{\mathbb{T} \setminus E} (Q_{N}*|f|^2)(\theta)\,d\theta \leq C \lambda \|f\|_1^2  
\end{align*}
where  $Q_N(x)=\frac{1}{N^{s-1}} \min(N^s, |x|^{-s})$.
\end{corollary}
\begin{proof}
Let $\|f\|_1=1$, $\lambda >1$ and $E:= \cup_I\; 5\cdot I$ defined as in Theorem \ref{first thing}.  For the moment let us assume $s>1$. We have already seen from Proposition \ref{2nd reduction} and Theorem \ref{first thing} that for $f^{(N)}=g^{(N)}+b^{(N)}_1+b_2^{(N)}$ we have
\begin{align*}
\int_{\mathbb{T} \setminus E} Q_{N}*|g^{(N)}|^2+\int_{\mathbb{T} \setminus E} Q_{N}*|b_2^{(N)}|^2 \lesssim \|Q_N\|_1 (\|g^{(N)}\|_2^2+ \|b_2^{(N)}\|_2^2) \lesssim \lambda \|Q_N\|_1.
\end{align*}
Since $\|Q_N\|_1 \lesssim 1$ uniformly in $N \in \mathbb{N}$ and $s>1$ we do not need to say more about $b_2^{(N)}$ and $g^{(N)}$.  For $b^{(N)}_1$, let $I^*$ be defined as in Theorem \ref{first thing}, $I^*=\frac{9}{8} \cdot I$. Then we have
\begin{align*}
&\int_{\mathbb{T} \setminus E} Q_{N}*|b_1^{(N)}|^2\lesssim \int_{\mathbb{T} \setminus E} Q_{N}*|b^*|^2+\int_{\mathbb{T} \setminus E} Q_{N}*|\tilde{b}|^2 \\
&\lesssim \int_{\mathbb{T} \setminus E} Q_{N}*|b^*|^2+ \|Q_N\|_1 \|\tilde{b}\|_2^2 \lesssim \int_{\mathbb{T} \setminus E} Q_{N}*|b^*|^2 + \lambda \|Q_N\|_1
\end{align*} 
Thus, the decay of $Q_N$ is all that is important, particularly for the bound
\begin{align*}
\int_{\mathbb{T} \setminus E} (Q_N*|b^*|^2)(x)\,dx = \sum_{i} \int_{\mathbb{T} \setminus E} \left(Q_N*\left|\sum_{I^* \subset \mathcal{C}_i}\chi_{I^*}f_I^{(N)}\right|^2\right)(x)\,dx
\end{align*}
For each individual $i$, we have
\begin{align*}
&\int_{\mathbb{T} \setminus E} \left(Q_N*\left|\sum_{I^* \subset \mathcal{C}_i}\chi_{I^*}f_I^{(N)}\right|^2\right)(x)\,dx \\
&=\int_{\mathbb{T} \setminus E} \int_{\mathbb{T}}Q_N(y)\left|\sum_{I^* \subset \mathcal{C}_i}\chi_{ I^*}(x-y) f^{(N)}_I(x-y)\right|^2\,dy dx
\end{align*}
By the argument from Theorem \ref{first thing},
\begin{align*}
&\int_{\mathbb{T} \setminus E} \int_{\mathbb{T}}Q_N(y)\left|\sum_{I^* \subset \mathcal{C}_i }\chi_{ I^*}(x-y) f^{(N)}_I(x-y)\right|^2\,dy dx\\
&\lesssim \frac{1}{N^{s-1}}\int_{\mathbb{T} \setminus c\cdot J_i} \int_{\mathbb{T}}|y|^{-s}\left(\chi_{4\cdot J^*_i}(x-y)\sum_{ I^* \subset \mathcal{C}_i } \left|f^{(N)}_I(x-y)\right|\right)^2\,dy dx \\
& \lesssim N^{2-s} \left(\sum_{I^* \subset \mathcal{C}_i }\|f_I\|_1\right)\int_{\mathbb{T} \setminus c\cdot J_i} \int_{\mathbb{T}} |y|^{-s}\chi_{4\cdot J^*_i}(x-y)\left(\sum_{I^* \subset \mathcal{C}_i } \left|f^{(N)}_I(x-y)\right|\right)\,dy dx \\
 &\lesssim N^{2-s}\left(\sum_{I^* \subset \mathcal{C}_i }\|f_I\|_1\right)\int_{\mathbb{T}}\chi_{4\cdot J^*_i}(y)\sum_{I^* \subset \mathcal{C}_i } |f^{(N)}_I(y)|\int_{\mathbb{T} \setminus c\cdot J_i}  |x-y|^{-s}\, dx dy 
\end{align*}
Then
\begin{align*}
&N^{2-s}\left(\sum_{I^* \subset \mathcal{C}_i }\|f_I\|_1\right)\int_{\mathbb{T}}\chi_{4\cdot J^*_i}(y)\sum_{I^* \subset \mathcal{C}_i } |f^{(N)}_I(y)|\int_{\mathbb{T} \setminus c\cdot J_i}  |x-y|^{-s}\, dx dy  \\
&\lesssim N^{2-s}\left(\sum_{I^* \subset \mathcal{C}_i }\|f_I\|_1\right)\left(\sum_{I^* \subset \mathcal{C}_i }\|f^{(N)}_I\|_1\right) \int_{\mathbb{T}} ||J_i| + x|^{-s} \,dx \\ 
&\lesssim \lambda \left(\sum_{I^* \subset \mathcal{C}_i }\|f_I\|_1\right) N^{2-s}|J_i| |J_i|^{-(s-1)} =\lambda \left(\sum_{I^* \subset \mathcal{C}_i }\|f_I\|_1\right) (N|J_i|)^{2-s}
\end{align*}
Here we see the importance of $s \geq 2$.  When $s\geq 2$ we can use the following inequality: $(N|J_i|)^{-1} < 1$ because $|J_i|>\frac{1}{N}$.   Then
\begin{align*}
&\int_{\mathbb{T} \setminus E} \left(Q_N*\left|\sum_{I^* \subset \mathcal{C}_i}\chi_{I^*}f_I^{(N)}\right|^2\right)(x)\,dx  \lesssim \lambda  \left(\sum_{I^* \subset \mathcal{C}_i }\|f_I\|_1\right)
\end{align*}
which is exactly what we need to obtain the final result. Otherwise, $(N|J_i|)^{2-s}$ grows as $N$ grows.
\end{proof}
The previous proof shows that any function $f \in L^1$ with $\mbox{supp}(\hat{f}) \subset [-N, N]$ with $\mathcal{B}_1 \neq \emptyset$ is a counterexample for $1<s<2$.  We now consider the analogue of Theorem \ref{first thing} on the line.
\begin{proposition} \label{real line}
Let  $\lambda>0$.  Then for any $f \in L^1(\mathbb{R})$, there exists $E \subset \mathbb{R}$, with $|E| \leq \frac{1}{\lambda}$ such that
\begin{align*}
\sup_{T>0}\frac{1}{T} \int_0^T \int_{\mathbb{R} \setminus E} |S_t f(x)|^2 dx dt \lesssim \lambda \|f\|_1^2\end{align*}
 where $S_tf=D_t * f$ given $D_t(x)=\int_{-t}^t e^{-2\pi i s x}\, ds$ .
\end{proposition}
\begin{proof}
Similar to the previous proposition, assume $\|f\|_1=1$, $f\geq 0$, fix $T$ and $\lambda>0$.  We also perform a Calder\'{o}n-Zygmund decomposition at height $\lambda$ and let $f= g+b_1+b_2$ as we do in Theorem \ref{first thing}, with $\mathcal{B}=\mathcal{B}_1\cup \mathcal{B}_2$ defined the same way.  We can define $E$ now as
\begin{align*}
E:= \bigcup_{I \in \mathcal{B}}5 \cdot I.
\end{align*}
 Note that
\begin{align*}
&D_t(x) = \int_{-t}^t e^{-2\pi i s x}\,ds =C\frac{e^{-2 \pi i t x}-e^{2 \pi i t x}}{x} \hspace{.3cm} \mbox{ for } |x|>T^{-1} \\
&\mbox{and } |D_t(x)| \leq T \hspace{.3cm} \mbox{for all } x.
\end{align*}
From this, we let $\chi_T(x)=\chi_{\{|x|>T^{-1}\}}$ and we assume as in Theorem \ref{first thing} that the Fourier support of $f$ is bounded. Therefore, we let $f=f^{(T)}:=f*V_T$ where $V_T$ is the real-line analogue to the de la Vall\'{e}e Poussin kernel.  Then we have
\begin{align} 
&\int_0^T \int_{\mathbb{R}\setminus E} |(\chi_T D_t*f^{(T)})(x)|^2\, dx dt \nonumber\\
& \lesssim \int_{\mathbb{R} \setminus E} \int_0^T \left| \int_{\mathbb{R}}f^{(T)}(x-y)\chi_T(y) \frac{e^{-2\pi i t y}}{y}\,dy \right|^2 +  \left| \int_{\mathbb{R}}f^{(T)}(x-y)\chi_T(y) \frac{e^{2\pi i t y}}{y}\,dy \right|^2\,dt dx  \nonumber \\
&\lesssim \int_{\mathbb{R} \setminus E} \int_{-T}^T \left|\hat{f_x}(t) \right|^2\,dt dx \lesssim \int_{\mathbb{R} \setminus E} \int_{\mathbb{R}} \left|\hat{f_x}(t) \right|^2\,dt dx \nonumber \\
&\lesssim \int_{\mathbb{R}\setminus E} (K^1 * |f^{(T)}|^2)(x)\,dx \label{realana0}
\end{align}
where $f_x(y)=f^{(T)}(x-y)\chi_T(y)\frac{1}{y}$ and $K^1(x) = \min(T^2, |x|^{-2})$.  Furthermore,
\begin{align} 
&\int_0^T \int_{\mathbb{R} \setminus E} |((1-\chi_T)D_t * f^{(T)})(x)|^2\,dx dt \nonumber\\
&\lesssim \int_0^T \int_{\mathbb{R} \setminus E} \int_{|y| \leq T^{-1}}T|f^{(T)}(x-y)|^2\,dy dx dt \nonumber\\
 &\lesssim \int_{\mathbb{R} \setminus E} (K^1 * |f^{(T)}|^2)(x) \,dx \label{realana}
 \end{align}
 Thus, by combining \eqref{realana0} and \eqref{realana} we conclude that
\begin{align*}
&\frac{1}{T} \int_0^T \int_{\mathbb{R} \setminus E} |(D_t*f)(x)|^2 \, dx dt \lesssim \frac{1}{T} \int_{\mathbb{R} \setminus E} (K^1 * |f^{(T)}|^2)(x) \,dx\\
&\lesssim \int_{\mathbb{R} \setminus E} (K_T * |f^{(T)}|^2)(x) \, dx 
\end{align*}
where $K_T(x):= T^{-1} \min(T^2, |x|^{-2})$ and $\|K_T\|_1 \lesssim 1$.  Then
\begin{align*}
\int_{\mathbb{R}\setminus E} &\int_{\mathbb{R}} |f^{(T)}(x-y)|^2K_T(y)\,dy dx \\
&\lesssim \int_{\mathbb{R}\setminus E} (|g^{(T)}|^2*K_T)(x)\,dx +\int_{\mathbb{R} \setminus E} (|b^{(T)}_1|^2*K_T)(x)\,dx +\int_{\mathbb{R}\setminus E} (|b^{(T)}_2|^2*K_T)(x)\,dx  
\end{align*}
and 
\begin{align*}
\int_{\mathbb{R}\setminus E} (|g^{(T)}|^2*K_T)(x)\,dx \lesssim \||g^{(T)}|^2*K_T\|_1\lesssim \|g^{(T)}\|_2^2\lesssim \lambda\|g\|_1 \lesssim \lambda 
\end{align*}
Now for the estimate for $b_1$, we perform a series of steps that are similar to what was done in Theorem \ref{first thing} with
\begin{align*}
b^{(T)}_1=\sum_{I \in \mathcal{B}_1} f^{(T)}_I =\sum_{I \in \mathcal{B}_1} \chi_{I^*}f^{(T)}_I + \sum_{I \in \mathcal{B}_1} \chi_{\mathbb{R}\setminus I^*}f^{(T)}_I=b^*+\tilde{b}
\end{align*}
with $I^*=\frac{9}{8}\cdot I$.  First considering $b^*$, we use Lemma \ref{lemma} to get a partition of $\cup_{I \in \mathcal{B}_1} I^*$ into intervals $\{C_i\}$ and
\begin{align*}
& \int_{\mathbb{R} \setminus E} |b^*(x-y)|^2*K_T(y)\,dy dx \\
& \hspace{1.5cm}=\sum_i \int_{\mathbb{R} \setminus E} \int_{\mathbb{R}}K_T(y)\left|\sum_{I^* \subset \mathcal{C}_i }\chi_{ I^*}(x-y) f^{(T)}_I(x-y)\right|^2\,dy dx 
\end{align*}
Then using the same argument as Theorem \ref{first thing} we have
\begin{align*}
&\sum_i \int_{\mathbb{R} \setminus E} \int_{\mathbb{R}}K_T(y)\left|\sum_{I^* \subset \mathcal{C}_i }\chi_{ I^*}(x-y) f^{(T)}_I(x-y)\right|^2\,dy dx  \\
&\lesssim \sum_i \int_{\mathbb{R} \setminus 5\cdot J_i} \int_{\mathbb{R}}K_T(y)\left(\chi_{4\cdot J_i}(x-y)\sum_{I^* \subset \mathcal{C}_i } \left|f^{(T)}_I(x-y)\right|\right)^2\,dy dx \\
&\lesssim \sum_i \sum_{I^* \subset \mathcal{C}_i} \|f_I\|_1 \int_{\mathbb{R} \setminus 5 \cdot J_i}\int_{\mathbb{R}}|y|^{-2}\left(\chi_{4\cdot J_i}(x-y)\sum_{I^* \subset \mathcal{C}_i } \left|f^{(T)}_I(x-y)\right|\right)\,dy dx \\
&\lesssim \sum_i \sum_{I^* \subset \mathcal{C}_i} \|f_I\|_1 \frac{\sum_{I^* \subset \mathcal{C}_i}\|f_I^{(T)}\|_1}{|J_i|} \lesssim \sum_i \lambda\sum_{I^* \subset \mathcal{C}_i}  \|f_I\|_1 \\
&\lesssim \lambda
\end{align*}

For this estimate, we do not have to worry about any convergence issues in the sum since $|\mathcal{B}_1| <\infty$. In Theorem \ref{first thing}, we showed that $\|\tilde{b}\|_{\infty} \lesssim \lambda$ which translates to $\mathbb{R}$ as well. However, for $b^{(T)}_2$ on the real line 
\begin{align*}
b^{(T)}_2= \sum_J K_T *(\chi_J b_2)= \sum_{J} f^{(T)}_J
\end{align*}
we have an infinite sum in $J$.  Using the same argument for $b_2$ as we did in Theorem \ref{first thing}, we obtain the following bound
\begin{align*}
\left\|b^{(T)}_2\right\|_{\infty} \lesssim \lambda \sum_{j=1}^{\infty} \frac{1}{j^2} \lesssim \lambda
\end{align*}
 Then, using this bound, we have
\begin{align*}
\int_{\mathbb{R}\setminus E} (|b^{(T)}_2|^2* K_T)(x)dx &\lesssim \lambda \int_{\mathbb{R}} (|b^{(T)}_2|*K_T)(x)dx \\
&\lesssim \lambda \|b^{(T)}\|_1\|K_T\|_1 \\
&\lesssim \lambda
\end{align*}
Therefore, we have the equivalent statement to Theorem \ref{first thing} for functions on $\mathbb{R}$.
\end{proof}
\subsection{Higher-Dimensional Results}
In order to demonstrate the process of extending these results to higher dimensions, we shall investigate the tensor case when for $f \in L^1(\mathbb{T}^d)$ we have
\begin{align*}
f(\theta_1, \theta_2, ..., \theta_d)= f_1(\theta_1)f_2(\theta_2)\cdots f_d(\theta_d)
\end{align*}
Drawing inspiration from Theorem \ref{first thing}, we will take the $\ell^2$ average over the first $N$ Fourier partial sums for each $f_i$:
\begin{align}\label{product}
\prod_{i=1}^d \left( \frac{1}{N} \sum_{n_i=1}^N|S_{n_i}f_i|^2 \right)
\end{align}
For $\overline{n}=(n_1, n_2, ..., n_d) \in \mathbb{N}^d$, let $\chi_{R_{\overline{n}}}$ be the indicator function on $\mathbb{Z}^d$ for the set $R_{\overline{n}}$ where
\begin{align*}
R_{\overline{n}}=\{ \overline{m} \in \mathbb{Z}^d : -n_j \leq m_j \leq n_j,\ 1\leq j \leq d\}
\end{align*}
Let $S_{\overline{n}}f$ be the Fourier multiplier operator on $f \in L^1(\mathbb{T}^d)$ defined by the Fourier multiplier $\chi_{R_{\overline{n}}}$. Then simplifying \eqref{product}, we get an average over $N^d$ partial Fourier sums corresponding to the lattice points $\overline{n}=(n_1, n_2, ..., n_d) \in (\mathbb{N}\setminus \{0\})^d$,
\begin{align*}
\prod_{i=1}^d \left( \frac{1}{N} \sum_{n_i=1}^N|S_{n_i}f_i|^2 \right)=\frac{1}{N^d} \sum_{\overline{n} \in R^+_N} |S_{\overline{n}}f|^2
\end{align*}
where $R^+_N=\{ \overline{n} \in (\mathbb{N} \setminus \{0\})^d : \|\overline{n}\|_{\infty} \leq N\}$.  Then, we execute a Calder\'{o}n-Zygmund decomposition at height $\lambda^{1/d}$ for each $f_i$.  Theorem \ref{first thing} informs us that there exists $E_i \subset \mathbb{T}$ such that $|E_i| \leq \frac{1}{\lambda^{1/d}}$ and 
\begin{align*}
\sup_{N \geq 1} \prod_{i=1}^d \left(  \int_{\mathbb{T} \setminus E_i}\frac{1}{N} \sum_{n_i=1}^N|S_{n_i}f_i|^2 \right) \lesssim \prod_{i=1}^d  \lambda^{1/d}\|f_i\|^2_1=\lambda \|f\|^2_{L^1(\mathbb{T}^d)}
\end{align*}
Define $E':= \cup_i \{ x \in \mathbb{T}^d : x_i \in E_i \}$. Then $\mathscr{H}^d(E') \leq \frac{1}{\lambda^{1/d}}$ and
\begin{align*}
\sup_{N\geq 1} \frac{1}{N^d} \sum_{\overline{n} \in R^+_N}\int_{\mathbb{T}^d \setminus E'}  |S_{\overline{n}}f|^2&=\sup_{N \geq 1}  \prod_{i=1}^d \left( \int_{\mathbb{T} \setminus E_i}\frac{1}{N} \sum_{n_i=1}^N|S_{n_i}f_i|^2 \right) \\
&\lesssim \lambda \|f\|^2_{L^1(\mathbb{T}^d)}
\end{align*}
The estimate $\mathscr{H}^d(E') \leq \lambda^{-1/d}$ is a step back from what we get in Theorem \ref{first thing}. This should be expected because the method by which we defined $E'$ was not as precise as the method for defining $E$ in Theorem \ref{first thing}. It is possible that there exists a subset of positive measure in $E'$ such that for each point in the subset one has $|f_i|>\lambda^{1/d}$ for at least one $j$, but $|f|\leq \lambda$. This implies that $E'$ is not the most efficient choice for the set we will remove from $\mathbb{T}^d$. Fortunately, we can take the $d$-dimensional Calder\'{o}n-Zygmund decomposition at height $\lambda$ for $f$ and still use the preceding ideas from Theorem \ref{first thing}.  The end result is that we can find an $E \subset \mathbb{T}^d$ with $\mathscr{H}^d(E) \leq \frac{1}{\lambda}$ such that the desired result holds.
\begin{proposition} \label{high d}
Let  $\lambda>0$.  Then for any $f \in L^1(\mathbb{T}^d)$, there exists $E \subset \mathbb{T}^d$, with $|E| \leq \frac{1}{\lambda}$ such that
\begin{align*}
\sup_{N \geq 1}\frac{1}{N^d} \sum_{\overline{n} \in R^+_N} \int_{\mathbb{T}^d \setminus E} |S_{\overline{n}}f(\theta)|^2 \,d\theta \leq C_d \lambda \|f\|_1^2
\end{align*}
where $R^+_N=\{ \overline{n} \in (\mathbb{N} \setminus \{0\})^d : \|\overline{n}\|_{\infty} \leq N\}$, $C_d$ is a constant depending only on the dimension $d$, and $S_{\overline{n}}f=D_{\overline{n}} * f$ given $D_{\overline{n}}(x):=\sum_{-n_j\leq m_j\leq n_j}e^{2\pi i m \cdot x}=\prod_{j=1}^d D_{n_j}(x_j)$.
\end{proposition}
\begin{proof}
Just as in Theorem \ref{first thing}, we perform a Calder\'{o}n-Zygmund decomposition at height $\lambda>0$ while fixing $N \geq 1$ and assuming $\|f\|_1=1$.  For this proof let $|I|:=\mathscr{H}^d(I)$ for any measurable set $I$.  We label the set of bad cubes $\mathcal{B}=\mathcal{B}_1 \cup \mathcal{B}_2$ where $|I|> N^{-d}$ for all $I \in \mathcal{B}_1$ and $|I|\leq N^{-d}$ for $I \in \mathcal{B}_2$ and 
\begin{align*}
E:=\bigcup_{I \in \mathcal{B}}5\cdot I.
\end{align*}
Then with $D_{n_1}$ acting in the first variable only and $S_{n_*}$ in all the others
\begin{align*}
&\sum_{n_i=1\atop 1\leq i \leq d}^N\int_{\mathbb{T}^d \setminus E} |S_{\overline{n}}f(\theta)|^2\,d\theta = \sum_{n_i=1 \atop i \geq 2}^N \sum_{n_1=1}^N \int_{\mathbb{T}^d \setminus E} |D_{n_1}*S_{n_*}f(\theta)|^2\,d\theta \\
& \lesssim\sum_{n_i=1 \atop i \geq 2}^N  \int_{\mathbb{T}^d \setminus E} \sum_{n_1=1}^N\left|\int_{|x_1|>N^{-1}}\frac{e(n_1x_1)}{\sin(x_1)}S_{n_*}f(\theta_1-x_1,\theta_*)\,dx_1\right|^2\,d\theta\\
&+ \sum_{n_i=1 \atop i \geq 2}^N  \int_{\mathbb{T}^d \setminus E}\int_{|x_1|<N^{-1}}N^2\left|S_{n_*}f(\theta_1-x_1,\theta_*)\right|^2\,dx_1\,d\theta
\end{align*}
where $\theta_*$ and $n_*$ are the vectors $\theta$ and $\overline{n}$, respectively, with the first component deleted. Now we use Plancherel as we did in Theorem \ref{first thing} with $S_{n_*}f$ in place of $f$.
\begin{align*}
&\sum_{n_i=1 \atop i \geq 2}^N  \int_{\mathbb{T}^d \setminus E} \sum_{n_1=1}^N\left|\int_{\mathbb{T}}\frac{e(n_1x_1)}{\sin(x_1)}S_{n_*}f(\theta_1-x_1,\theta_*)\,dx_1\right|^2\,d\theta \\
& \lesssim \sum_{n_i=1 \atop i \geq 2}^N  \int_{\mathbb{T}^d \setminus E}(K_{1,N}*\left|S_{n_*}f(\theta_*)\right|^2)(\theta_1)\,d\theta
\end{align*}
where $K_{1,N}(z)=\mbox{min}(N^2,|z|^{-2})$ for $z \in \mathbb{T}$.  Now we note that
\begin{align*}
&\sum_{n_i=1 \atop i \geq 2}^N  \int_{\mathbb{T}^d \setminus E}(K_{1,N}*\left|S_{n_*}f(\theta_*)\right|^2)(\theta_1)\,d\theta \\
&=\sum_{n_i=1 \atop i \geq 3}^N  \int_{\mathbb{T}^d \setminus E}(K_{1,N}*\sum_{n_2=1}^N\left|S_{n_*}f(\theta_*)\right|^2)(\theta_1)\,d\theta.
\end{align*}
Fixing $\theta_1$, we have 
\begin{align*}
&\sum_{n_2=1}^N |S_{n_*}f(\theta)|^2 \\
&\lesssim \sum_{n_1=1}^N\left|\int_{|x_2|>N^{-1}}\frac{e(n_2x_2)}{\sin(x_2)}S_{n_{**}}f(\theta_1, \theta_2-x_2,\theta_{**})\,dx_2\right|^2 \\
& \hspace{1.5cm}+ \sum_{n_1=1}^N\left|\int_{|x_2|<N^{-1}}NS_{n_{**}}f(\theta_1, \theta_2-x_2,\theta_{**})\,dx_2\right|^2\\
& \lesssim (K_{1,N}*\left|S_{n_{**}}f(\theta_1, \theta_{**})\right|^2)(\theta_2).
\end{align*}
where $n_{**}$ and $\theta_{**}$ are the vectors $\overline{n}$ and $\theta$, respectively, with the first and second components deleted. Repeating this process leaves us with
\begin{align*}
 \sum_{\overline{n} \in R^+_N} \int_{\mathbb{T}^d \setminus E} |S_{\overline{n}}f(\theta)|^2 \,d\theta \lesssim N^d \int_{\mathbb{T}^d\setminus E}  (K_{(d),N} * |f^{(N)}|^2)(\theta)  \,d\theta 
\end{align*}
where 
\begin{align*}
K_{(d),N}(\theta) := \frac{1}{N^d} \prod_{i=1}^d K_1(\theta_i) =\frac{1}{N^d} \prod_{i=1}^d \min(N^2,|\theta_i|^{-2}).
\end{align*}
Here we also take the liberty to replace $f$ with $f^{(N)}=V_{d,N}*f$, where $V_{d,N}(x):=\prod_{i=1}^d V_N(x_i)$. Now by the same process as Theorem \ref{first thing} we have
\begin{align*}
&\int_{\mathbb{T}^d\setminus E}  K_{(d),N} * |f^{(N)}|^2  \\
&\lesssim \int_{\mathbb{T}^d\setminus E}  K_{(d),N} * |g^{(N)}|^2 +\int_{\mathbb{T}^d\setminus E}  K_{(d),N} * |b^{(N)}_1|^2 +\int_{\mathbb{T}^d\setminus E}  K_{(d),N} * |b^{(N)}_2|^2
\end{align*}
where, in this case, $\mathcal{B}_1$ is the set of bad cubes of length $>1/N$ and thus of measure $>N^{-d}$. Then for $g^{(N)}$, $\|g^{(N)}\|_{\infty} \lesssim \lambda$ and 
\begin{align*}
\int_{\mathbb{T}^d\setminus E}  K_{(d),N} * \left|g^{(N)}\right|^2  \lesssim \lambda \|K_{(d),N}*g^{(N)}\|_1 \lesssim \lambda \|K_{(d),N}\|_1\|g^{(N)}\|_1 \lesssim \lambda.
\end{align*}
For any cube $I$, let $I^*:= \frac{9}{8} \cdot I$, and let $b^{(N)}_1=b^*+\tilde{b}= \sum_{I \in \mathcal{B}_1} \chi_{I^*}f_I^{(N)}+\sum_{I \in \mathcal{B}_1} \chi_{\mathbb{T}^d \setminus I^*}f_I^{(N)}$.  We must account for the fact that the set $I^*$ are not pairwise disjoint, so we will utilize Lemma \ref{lemma2} to deal with this.  First we break up $\cup_{I \in \mathcal{B}_1} I^*$ into pairwise disjoint components $\{\mathcal{C}_i\}$
\begin{align*}
&\int_{\mathbb{T}^d\setminus E}  K_{(d),N} * |b^*|^2  \\
&= \sum_{i} \int_{\mathbb{T}^d\setminus E} K_{(d),N} * \left| \sum_{I^* \subset \mathcal{C}_i}\chi_{I^*}f^{(N)}_I\right|^2 \lesssim \sum_{i} \int_{\mathbb{T}^d\setminus E} K_{(d),N} * \left( \sum_{I^* \subset \mathcal{C}_i}\chi_{I^*}\left|f^{(N)}_I\right|\right)^2 \\
&\lesssim \sum_{i} \left(\sum_{I \in \mathcal{B}_1 \atop I^* \subset \mathcal{C}_i}\|f_I\|_1\right)N^d\int_{\mathbb{T}^d\setminus E}  K_{(d),N} * \sum_{I^* \subset \mathcal{C}_i}\chi_{I^*}\left|f^{(N)}_I\right| \hspace{.5cm} \mbox{By Bernstein} 
\end{align*}
It is here we use the same argument as Theorem \ref{first thing}.  Specifically, we have put ourselves in a position to take advantage of the decay of the kernel $K_{(d),N}$ due to our careful definition of $E$ and $\mbox{supp}(\sum_{I^* \subset \mathcal{C}_i} \chi_{I^*}\left|f^{(N)}_I\right|)$.   For each $i$, let $J_i$ be the largest cube such that $J_i \subset \mathcal{C}_i$, then
\begin{align*}
\int_{\mathbb{T}^d\setminus E}  K_{(d),N} * \sum_{I^* \subset \mathcal{C}_i} \chi_{I^*}\left|f^{(N)}_I\right| &\lesssim \int_{\mathbb{T}^d\setminus 5\cdot J_i}  \int_{\mathbb{T}^d} K_{(d),N}(x-y)\chi_{4\cdot J_i}(y)\left|\sum_{I^* \subset \mathcal{C}_i} f^{(N)}_I(y)\right|\,dy dx\\
&\lesssim \left(\sum_{I^* \subset \mathcal{C}_i}\|f_I\|_1\right) \left(\sup_{y \in 4\cdot J_i} \int_{\mathbb{T}^d \setminus 5\cdot J_i} K_{(d),N}(x-y)\,dx\right) \\
&\lesssim \left(\sum_{I^* \subset \mathcal{C}_i}\|f_I\|_1\right)N^{-d}(|J_i|)^{-1}
\end{align*}
Since $\cup_{I^* \subset \mathcal{C}_i} I \subset 4\cdot J_i$, $\sum |I| \lesssim |J_i|$. Therefore, $(\sum |I|)/|J_i| \lesssim 1$ and
\begin{align*}
\int_{\mathbb{T}^d\setminus E}  K_{(d),N} * \sum_{I^* \subset \mathcal{C}_i} \chi_{I^*}\left|f^{(N)}_I\right| &\lesssim \left(\sum_{I^* \subset \mathcal{C}_i}\|f_I\|_1\right)N^{-d}(|J_i|)^{-1}= \left(\sum_{I^* \subset \mathcal{C}_i}|I|\frac{\|f_I\|_1}{|I|}\right)N^{-d}(|J_i|)^{-1}\\
&\lesssim \left(\sum_{I^* \subset \mathcal{C}_i}|I|\lambda \right)N^{-d}(|J_i|)^{-1} \\
&\lesssim \lambda N^{-d}.
\end{align*} Using this estimate we have
\begin{align*}
\int_{\mathbb{T}^d\setminus E}  K_{(d),N} * |b^*|^2  &\lesssim \sum_{i} \left(\sum_{I \in \mathcal{B}_1 \atop I^* \subset \mathcal{C}_i}\|f_I\|_1\right)N^d\int_{\mathbb{T}^d\setminus E}  K_{(d),N} *\left| \sum_{I^* \subset \mathcal{C}_i}\chi_{I^*}f^{(N)}_I\right| \\
&\lesssim \sum_{i} \left(\sum_{I \in \mathcal{B}_1 \atop I^* \subset \mathcal{C}_i}\|f_I\|_1\right)N^d \left(N^{-d} \lambda\right) \lesssim \lambda \sum_{I \in \mathcal{B}_1} \|f_I\|_1 \\
& \lesssim \lambda
\end{align*}
which is what we need for $b^*$.  Just as in Theorem \ref{first thing}, we show that $\|\tilde{b}\|_{\infty} \lesssim \lambda$.  The key in the translation to this setting is that we still retain the property that the cubes $I \in \mathcal{B}_1$ are disjoint and that the distance between each interval $I$ and $\mathbb{T}^d\setminus \cup_{I \in \mathcal{B}_1} I^*$ is at least a constant factor of the length of $I$.  Therefore, we can say that
\begin{align*}
&\left|\tilde{b}(x)\right|\lesssim \sum_{I \in \mathcal{B}_1} \chi_{\mathbb{T}^d \setminus I^*}(x)\left|f_I^{(N)}(x)\right|=\sum_{I \in \mathcal{B}_1} \chi_{\mathbb{T}^d \setminus I^*}(x)\left|\int_{\mathbb{T}^d} V_{d,N}(x-y)f_I(y)\,dy\right|.
\end{align*}
Assuming $x \in \mathbb{T}^d \setminus \cup_{I \in \mathcal{B}_1} I^*$, we have
\begin{align*}
\sum_{I \in \mathcal{B}_1} \chi_{\mathbb{T}^d \setminus I^*}(x)\left|\int_{\mathbb{T}^d} V_{d,N}(x-y)f_I(y)\,dy\right|& \lesssim \sum_{I \in \mathcal{B}_1} \int_{\mathbb{T}^d} \left|\chi_{I}(y)V_{d,N}(x-y)f_I(y)\right|\,dy \\
&\lesssim \sum_{I \in \mathcal{B}_1}  \|f_I\|_1 \sup_{y \in I }\left|\chi_I(y)V_{d,N}(x-y)\right|
\end{align*}
We again establish that for $x \in \mathbb{T}^d \setminus I^*$
\begin{align*}
\sup_{y \in I } |x-y| \gtrsim \max\left(\mbox{dist}(x, I), \frac{9}{16} (\mbox{length of }I)\right) \gtrsim \mbox{dist}(x, I)+ (\mbox{length of }I)
\end{align*}
Then using the properties of the Calder\'{o}n-Zygmund decomposition, 
\begin{align*}
\sum_{I \in \mathcal{B}_1}  \|f_I\|_1 \sup_{y \in I }\left|\chi_I(y)V_{d,N}(x-y)\right|&=\sum_{I \in \mathcal{B}_1}  \|f_I\|_1 \frac{|I|}{|I|}\sup_{y \in I }\left|\chi_I(y)V_{d,N}(x-y)\right| \\
&\lesssim \sum_{I \in \mathcal{B}_1}  \|f_I\|_1 \frac{1}{|I|}\int_{\mathbb{T}^d} \chi_{I}(z)\sup_{y \in I }\left|\chi_I(y)V_{d,N}(x-y) \right|\,dz\\
& \lesssim \lambda\sum_{I \in \mathcal{B}_1}  \int_{\mathbb{T}^d}  \chi_{I}(z)\sup_{y \in I }\left|\chi_I(y)V_{d,N}(x-y) \right|\,dz
\end{align*}
Then by the same Riemann sum argument from Theorem \ref{first thing},
\begin{align*}
& \lambda\sum_{I \in \mathcal{B}_1}  \int_{\mathbb{T}^d}  \chi_{I}(z)\sup_{y \in I }\left|\chi_I(y)V_{d,N}(x-y) \right|\,dz \\
&\lesssim \lambda\sum_{I \in \mathcal{B}_1}  \int_{\mathbb{T}^d} \chi_{I}(z)\left|V_{d,N}( \mbox{dist}(x, I)+ (\mbox{length of }I)) \right|\,dz \\
&\lesssim \lambda \int_{\mathbb{T}^d} \sum_{I \in \mathcal{B}_1}\chi_I(z)\left|V_{d,N}( \mbox{dist}(x, I)+ (\mbox{length of }I))\right|\,dz \\
&\lesssim \lambda \int_{\mathbb{T}^d} |V_{d,N}(x-z)|\,dz \lesssim \lambda
\end{align*}
  Thus
\begin{align*}
&\int_{\mathbb{T}^d\setminus E}  K_{(d),N} * \left|b^{(N)}_1\right|^2  \\
&\lesssim \int_{\mathbb{T}^d\setminus E}  K_{(d),N} * \left|b^*\right|^2+\int_{\mathbb{T}^d\setminus E}  K_{(d),N} * \left|\tilde{b}\right|^2\\
&\lesssim  \lambda
\end{align*}
Finally, for $b^{(N)}_2$, we will show that $\|b^{(N)}_2\|_{\infty} \lesssim \lambda$
\begin{align*}
\left|b^{(N)}_2(x)\right| &\lesssim \sum_{J} \left|f^{(N)}_J(x)\right| \lesssim \sum_{J} \int_{\mathbb{T}^d}|f_J(y)||V_{d,N}(x-y)|\,dy\\
&\lesssim \sum_J \|f_J\|_1\sup_{y \in J} |V_{d,N}(x-y)| \\
&\lesssim \lambda \sum_{J} \frac{1}{N^d} \sup_{y \in J} |V_{d,N}(x-y)|
\end{align*}
Now, we can bound this sum by a Riemann sum just as we did for $\tilde{b}$ or $b^{(N)}_2$ from Theorem \ref{first thing}.  Therefore, we have
\begin{align*}
 \lambda \sum_{J} \frac{1}{N^d} \sup_{y \in J} |V_{d,N}(x-y)| &\lesssim \lambda \int_{\mathbb{T}^d} |V_{d,N}(x-z)|\,dz \\
& \lesssim \lambda 
\end{align*}
In summary,
\begin{align*}
&\frac{1}{N^d} \sum_{\overline{n} \in R^+_N} \int_{\mathbb{T}^d \setminus E} |S_{\overline{n}}f(\theta)|^2 \,d\theta \\
& \lesssim \int_{\mathbb{T}^d\setminus E}  (K_{(d),N} * |f^{(N)}|^2)(\theta)  \,d\theta \\
&\lesssim \int_{\mathbb{T}^d\setminus E}  K_{(d),N} * |g^{(N)}|^2 +\int_{\mathbb{T}^d\setminus E}  K_{(d),N} * |b^{(N)}_1|^2 +\int_{\mathbb{T}^d\setminus E}  K_{(d),N} * |b^{(N)}_2|^2 \\
&\lesssim \lambda
\end{align*}
as desired.
\end{proof}
%
%
%
%
%
%
%
%
\subsection{Higher Powers}
Let us consider the statement analogous to Theorem \ref{first thing} with $p>2$ instead of $p=2$:
\begin{align*}
\sup_{N\geq 1} \frac{1}{N} \sum_{n=1}^N \int_{\mathbb{T} \setminus E} |S_nf(\theta)|^p\,d\theta \lesssim \lambda^{p-1}\|f\|_1^p
\end{align*}
where  for all $f \in L^1(\mathbb{T})$, $|E|\leq \frac{1}{\lambda}$. We have reasonable suspicion to believe that this estimate may hold given Theorem \ref{strong sums} and Theorem \ref{first thing}.  However, we also may expect to obtain a contradicting statement if indeed we encounter growth in the integral of $\frac{1}{N}\sum_{n=1}^N |S_nf|^p$ on $\mathbb{T} \setminus E$ for $p>2$ similar to the growth of the integral of 
\begin{align*}
 \frac{1}{N} \sum_{n=1}^N n^p \chi_{[-\frac{1}{N}, \frac{1}{N}]}
\end{align*}
 on $\mathbb{T}$ for $p>1$. By expanding the methods used for Theorem \ref{first thing} we can provide more evidence indicating the latter to be true:

\begin{proposition} \label{p>2}
Let $\lambda >0$. Then for $f \in L^1(\mathbb{T})$, $p\geq 2$, there exists $E \subset \mathbb{T}^d$, with $|E| \leq \frac{1}{\lambda}$ such that
\begin{align}
\sup_{N \geq 1} \ (N\log^{p-2} N)^{-1} \sum_{n=1}^N \int_{\mathbb{T}\setminus E} |S_nf(\theta)|^p  \,d\theta \lesssim \lambda^{p-1} \|f\|^p_1
\end{align}
\end{proposition}

\begin{proof}
We first prove that this estimate works for $p=2m$, $m \in \mathbb{N}$, $m \geq 2$.  An interpolation argument provides the rest of the result.  First, we let $\|f\|_1=1$, $f\geq 0$, and $E=\cup_{I \in \mathcal{B}}\; 5\cdot I$ where $\mathcal{B}$ is the set of bad intervals from a Calder\'{o}n-Zygmund decomposition of $f \in L^p$ at height $\lambda$ for $\lambda>1$.
First we assume $m=2$, then we consider the functions $g$ and $b_2$ using the same notation as in Theorem \ref{first thing}, where $f=g+b_1+b_2$.  Then, we note that
\begin{align*} &\frac{1}{N} \sum_{n=1}^N \int_{\mathbb{T}\setminus E} |S_nf(x)|^4\,dx \\
 &\lesssim \frac{1}{N} \sum_{n=1}^N \int_{\mathbb{T}\setminus E} |S_ng(x)|^4\,dx \\
& \hspace{.5cm}+\frac{1}{N} \sum_{n=1}^N \int_{\mathbb{T}\setminus E} |S_nb_1(x)|^4\,dx+\frac{1}{N} \sum_{n=1}^N \int_{\mathbb{T}\setminus E} |S_nb_2(x)|^4\,dx.
\end{align*} 
By the $L^4$ boundedness of the Hilbert Transform,
\begin{align*}
\frac{1}{N} \sum_{n=1}^N \int_{\mathbb{T}\setminus E} |S_ng(x)|^4\,dx &\lesssim \frac{1}{N} \sum_{n=1}^N \int_{\mathbb{T}} |S_ng(x)|^4\,dx =\frac{1}{N} \sum_{n=1}^N\|S_ng\|_4^4 \\
&\lesssim \frac{1}{N} \sum_{n=1}^N\|g\|_4^4 \lesssim \lambda^3 \|g\|_1\\
&\lesssim \lambda^3
\end{align*}
We now make use of the fact that $f$ is less than a constant factor of $\lambda$ on the small intervals. We also note that for any function $f$ on $\mathbb{T}$, $f^{(N)}:=f*V_N$, just as in Theorem \ref{first thing}. Using the Hilbert transform again, we have
\begin{align*}
\frac{1}{N} \sum_{n=1}^N \int_{\mathbb{T}\setminus E} |S_nb_2(x)|^4dx &=\frac{1}{N} \sum_{n=1}^N \int_{\mathbb{T}} |S_nb^{(N)}_2(x)|^4dx\\
&\lesssim \frac{1}{N} \sum_{n=1}^N\|b_2^{(N)}\|_4^4 \lesssim \lambda^3 \|b_2^{(N)}\|_1 \\
&\lesssim \lambda^3
\end{align*}
Now for the large intervals we start by using the same process as before.  We again note that $D_n= k_1+k_2$, where $k_1(x)=\chi_{[-\frac{1}{N},\frac{1}{N}]}(x) D_n(x)$ and $k_2(x) = \chi_{[-\frac{1}{N}, \frac{1}{N}]^c}(x) D_n(x)=C \frac{e(nx)-e(-nx)}{\sin(x)}\chi_{[-\frac{1}{N},\frac{1}{N}]^c}(x)$.  Then 
\begin{align*}
\frac{1}{N} \sum_{n=1}^N \int_{\mathbb{T}\setminus E} |S_nb_1(x)|^4dx&=\frac{1}{N} \sum_{n=1}^N \int_{\mathbb{T}\setminus E} |S_nb^{(N)}_1(x)|^4dx\\
&\lesssim \frac{1}{N} \sum_{n=1}^N \int_{\mathbb{T}\setminus E} |(k_1*b^{(N)}_1)(x)|^4dx\\
& \hspace{.5cm}+\frac{1}{N} \sum_{n=1}^N \int_{\mathbb{T}\setminus E} |k_2*b^{(N)}_1(x)|^4dx
\end{align*}
and we first attack $k_1*b^{(N)}_1$:
\begin{align*}
\frac{1}{N}\sum_{n=1}^N \int_{\mathbb{T}\setminus E} |(k_1*b^{(N)}_1)(x)|^{4} \,dx &= \sum_{n=1}^N \int_{\mathbb{T}\setminus E} \left| \int_{\mathbb{T}} \chi_{[-\frac{1}{N},\frac{1}{N}]}(y)D_n(y) b^{(N)}_1(x-y)dy \right|^{4} \,dx 
\end{align*}
$\chi_{[-\frac{1}{N},\frac{1}{N}]}(y)D_n(y)dy$ are measures of uniformly bounded mass because $|D_n| \lesssim n$, so by Jensen
\begin{align*}
 \sum_{n=1}^N \int_{\mathbb{T}\setminus E} \left| \left(\chi_{[-\frac{1}{N},\frac{1}{N}]}D_n* b^{(N)}_1\right)(x) \right|^{4} \,dx &\lesssim \sum_{n=1}^N \int_{\mathbb{T}\setminus E} \left(\chi_{[-\frac{1}{N},\frac{1}{N}]}|D_n|* |b^{(N)}_1|^4\right)(x) \,dx\\
&\lesssim \sum_{n=1}^N n\int_{\mathbb{T}\setminus E} \left(\chi_{[-\frac{1}{N},\frac{1}{N}]}* |b^{(N)}_1|^4\right)(x) \,dx\\
&\lesssim N^2\int_{\mathbb{T}\setminus E} \left(\chi_{[-\frac{1}{N},\frac{1}{N}]}* |b^{(N)}_1|^4\right)(x) \,dx\\
\end{align*}
By the decomposition of $b_1^{(N)}$ from Theorem \ref{first thing}, $b_1^{(N)}=b^*+\tilde{b}= \sum_{I \in \mathcal{B}_1} \chi_{I^*} f_I^{(N)}+ \sum_{I \in \mathcal{B}_1} \chi_{\mathbb{T} \setminus I^*} f_I^{(N)}$
\begin{align*}
&N^2 \int_{\mathbb{T}\setminus E}  \left(\chi_{[-\frac{1}{N},\frac{1}{N}]}* \left|b^{(N)}_1\right|^4\right)(x) \,dx \\
&\lesssim  N^2 \int_{\mathbb{T}\setminus E} \left(\chi_{[-\frac{1}{N},\frac{1}{N}]} *\left|b^*\right|^4\right)(x) \,dx+ N^2 \int_{\mathbb{T}\setminus E} \int_{\mathbb{T}} \left(\chi_{[-\frac{1}{N},\frac{1}{N}]}* \left|\tilde{b}\right|^4\right)(x) \,dy dx
\end{align*}
We first note that $\mbox{supp}(b^*)= \cup_{I \in \mathcal{B}_1} I^*$ implies $\mbox{supp}(|b^*|^4)= \cup_{I \in \mathcal{B}_1} I^*$, and 
\begin{align*}
\mbox{supp} \left(\chi_{[-\frac{1}{N},\frac{1}{N}]} *\left|b^*\right|^4\right) = [\frac{-1}{N},\frac{1}{N}] + \bigcup_{I \in \mathcal{B}_1} I^* \subset \bigcup_{I \in \mathcal{B}_1} 5\cdot I =E 
\end{align*}
Therefore,
\begin{align*}
N^2 \int_{\mathbb{T}\setminus E} \left(\chi_{[-\frac{1}{N},\frac{1}{N}]} *\left|b^*\right|^4\right)(x) \,dx=0
\end{align*}
and from the proof of Theorem \ref{first thing} we know that $\|\tilde{b}\|_{\infty} \lesssim \lambda$, so
\begin{align*}
N^2 \int_{\mathbb{T}\setminus E}  \left(\chi_{[-\frac{1}{N},\frac{1}{N}]}* \left|\tilde{b}\right|^4\right)(x) \,dx  &\lesssim N^2 \lambda^3 \int_{\mathbb{T}\setminus E}  \left(\chi_{[-\frac{1}{N},\frac{1}{N}]}* \left|\tilde{b}\right|\right)(x) \, dx\\
&\lesssim N^2\lambda^3 \frac{1}{N} \|\tilde{b}\|_1 \\
& \lesssim \lambda^3N 
\end{align*}
For $k_2$, if we suppress the $\chi_{[-\frac{1}{N},\frac{1}{N}]^c}$ for the moment, we have 
\begin{align*}
\sum_{n=1}^N \int_{\mathbb{T}\setminus E} |k_2*b^{(N)}_1(x)|^{4} \,dx &\lesssim \sum_{n=1}^N \int_{\mathbb{T}\setminus E} \left| \int_{\mathbb{T}} e(ny)\frac{b^{(N)}_1(x-y)}{\sin(y)} dy \right|^{4} \,dx \\
&\lesssim  \int_{\mathbb{T}\setminus E}\sum_{n \in \mathbb{Z}} \left| \int_{\mathbb{T}} e(ny) \frac{b^{(N)}_1(x-y)}{\sin(y)} dy \right|^{4} \,dx \\
&\lesssim  \int_{\mathbb{T}\setminus E} \sum_{n\in \mathbb{Z}}\left| \mathcal{F}\left(\frac{b^{(N)}_1(x - \cdot)}{\sin(\cdot)}\right)(n) \right|^4 \,dx
\end{align*}
where $\mathcal{F}$ denotes the Fourier Transform. Now we use the identity $\hat{f}(n)\hat{g}(n)= \widehat{f*g}(n)$ and Plancherel  
\begin{align*}
& \int_{\mathbb{T}\setminus E} \sum_{n\in \mathbb{Z}}\left| \mathcal{F}\left(\frac{b^{(N)}_1(x - \cdot)}{\sin(\cdot)}\right)(n) \right|^4 \,dx\\
&= \int_{\mathbb{T}\setminus E} \sum_{n\in \mathbb{Z}}\left| \mathcal{F}\left(\frac{b^{(N)}_1(x - \cdot)}{\sin(\cdot)}*\frac{b^{(N)}_1(x - \cdot)}{\sin(\cdot)}\right)(n) \right|^2 \,dx\\
&= \int_{\mathbb{T}\setminus E} \int_{\mathbb{T}}\left| \left(\frac{b^{(N)}_1(x - \cdot)}{\sin(\cdot)}*\frac{b^{(N)}_1(x - \cdot)}{\sin(\cdot)}\right)(y)\right|^2 \,dy dx 
\end{align*}
We note here that the convolution of $\chi_{[\frac{-1}{N},\frac{1}{N}]^c}(z)\frac{b_1^{(N)}(x-z)}{\sin(z)}$ with itself leaves us with no ability to bound the set on which we expect the convolution to be large independently of $N$, namely $\cup_{I \in \mathcal{B}_1}I +\cup_{I \in \mathcal{B}_1}I$ (for example, let $\{I\}$ be a set of intervals decreasing in size geometrically, then if $N\gg \lambda$, $\cup_{I \in \mathcal{B}_1}I +\cup_{I \in \mathcal{B}_1}I = \mathbb{T}$). We now use Young's inequality to bound the convolution
\begin{align*}
&\int_{\mathbb{T}\setminus E} \int_{\mathbb{T}}\left| \left(\frac{b^{(N)}_1(x - \cdot)}{\sin(\cdot)}*\frac{b^{(N)}_1(x - \cdot)}{\sin(\cdot)}\right)(y)\right|^2 \,dy dx\\
&\lesssim \int_{\mathbb{T}\setminus E} \left(\int_{\mathbb{T}} k_3(y_1)|b^{(N)}_1(x-y_1)| \,dy_1 \right)^2 \int_{\mathbb{T}} k_3^2(y_2)|b^{(N)}_1(x-y_2)|^2\,dy_2dx
 \end{align*}
where $k_3(x):=\chi_{[-\frac{1}{N},\frac{1}{N}]^c}(x)|x|^{-1}$.  We know from Corollary \ref{decay corollary} that  
\begin{align*}
\int_{\mathbb{T}\setminus E}  \int_{\mathbb{T}} k_3^2(y_2)|b^{(N)}_1(x-y_2)|^2\,dy_2dx \lesssim N\lambda,
\end{align*} 
so we look to show $\sup_{x \in \mathbb{T}\setminus E}\int_{\mathbb{T}} k_3(y_1)|b^{(N)}_1(x-y_1)| \,dy_1 \lesssim \lambda \log N$.  For any $x  \in \mathbb{T} \setminus E$,
\begin{align*}
\int_{\mathbb{T}} k_3(x-y)|b^{(N)}_1(y)|\,dy &\lesssim  \int_{\mathbb{T}} k_3(x-y)|b^*(y)|\,dy+\int_{\mathbb{T}} k_3(x-y)|\tilde{b}(y)|\,dy \\
&\lesssim \int_{\mathbb{T}} k_3(x-y)|b^*(y)|\,dy+\lambda \int_{\mathbb{T}} k_3(x-y)\,dy \\
&\lesssim  \int_{\mathbb{T}} k_3(x-y)|b^*(y)|\,dy+\lambda\log N.
\end{align*}
For $b^*$,
\begin{align*}
\int_{\mathbb{T}} k_3(x-y)|b^*(y)|\,dy \lesssim \sum_i \int_{\mathbb{T}} k_3(x-y)\left|\sum_{I \in \mathcal{B}_1 \atop I \subset \mathcal{C}_i} \chi_{I^*}(y)f^{(N)}_I(y)\right|\,dy.
\end{align*}
Just as we have done before, we use the fact that $x \in \mathbb{T} \setminus E$ and $y \in \cup_i \mathcal{C}_i= \cup_I I^*$ and therefore $|x-y|\gtrsim |J_i| \gtrsim |\mathcal{C}_i|$.  Thus,
\begin{align*}
&k_3(x-y) \lesssim \frac{1}{\mbox{max}(\mbox{dist}(x, \mathcal{C}_i), |\mathcal{C}_i|)} \lesssim \frac{1}{\mbox{dist}(x, \mathcal{C}_i)+ |\mathcal{C}_i|} \hspace{.5cm} \mbox{when } y \in \mathcal{C}_i
\end{align*}
and therefore
\begin{align*}
\sum_i \int_{\mathbb{T}} k_3(x-y) \left|\sum_{I \in \mathcal{B}_1 \atop I^* \subset \mathcal{C}_i} \chi_{I^*}(y) f_I^{(N)}(y)\right|\,dy &\lesssim \sum_i \int_{\mathbb{T}} k_3(x-y) \sum_{I \in \mathcal{B}_1 \atop I^* \subset \mathcal{C}_i} \chi_{I^*}(y) \left|f_I^{(N)}(y)\right|\,dy \\
&\lesssim \sum_i \int_{\mathbb{T}} k_3(x-y) \chi_{\mathcal{C}_i}(y) \sum_{I \in \mathcal{B}_1 \atop I^* \subset \mathcal{C}_i} \left|f_I^{(N)}(y)\right|\,dy \\
&\lesssim \sum_i   \frac{1}{\mbox{dist}(x, \mathcal{C}_i)+ |\mathcal{C}_i|}  \sum_{I \in \mathcal{B}_1 \atop I^* \subset \mathcal{C}_i} \left\|f_I^{(N)}\right\|_1 \\
&\lesssim \sum_i   \frac{|\mathcal{C}_i|}{\mbox{dist}(x, \mathcal{C}_i)+ |\mathcal{C}_i|} \frac{1}{|\mathcal{C}_i|} \sum_{I \in \mathcal{B}_1 \atop I^* \subset \mathcal{C}_i} |I|\frac{\|f_I\|_1}{|I|}.
\end{align*}
Now using the properties of the $f_I$ and the fact that the $\mathcal{C}_i$ are disjoint, we get
\begin{align*}
\sum_i   \frac{|\mathcal{C}_i|}{\mbox{dist}(x, \mathcal{C}_i)+ |\mathcal{C}_i|} \frac{1}{|\mathcal{C}_i|} \sum_{I \in \mathcal{B}_1 \atop I^* \subset \mathcal{C}_i} |I|\frac{\|f_I\|_1}{|I|} &\lesssim \lambda \sum_i   \frac{|\mathcal{C}_i|}{\mbox{dist}(x, \mathcal{C}_i)+ |\mathcal{C}_i|} \frac{\sum_{I \in \mathcal{B}_1 \atop I^* \subset \mathcal{C}_i} |I|}{|\mathcal{C}_i|} \\
&\lesssim \lambda \sum_i   \frac{|\mathcal{C}_i|}{\mbox{dist}(x, \mathcal{C}_i)+ |\mathcal{C}_i|} \\
& \lesssim \lambda \int_{\mathbb{T}} k_3(x-y)\,dy \lesssim \lambda \log N
\end{align*} 
From this we have
\begin{align*}
&\int_{\mathbb{T}} k_3(x-y)|b^{(N)}_1(y)|\,dy  \lesssim \int_{\mathbb{T}} k_3(x-y)|b^*(y)|\,dy+\lambda\log N \\
&\lesssim \lambda \log N
\end{align*}
Combining this with what we have already show for $b_1^{(N)}$, we have
\begin{align*}
\sum_{n=1}^N \int_{\mathbb{T}\setminus E} |k_2*b^{(N)}_1(x)|^{4}\,dx &\lesssim \int_{\mathbb{T}\setminus E} \int_{\mathbb{T}}\left| \left(\frac{b^{(N)}_1(x - \cdot)}{\sin(\cdot)}*\frac{b^{(N)}_1(x - \cdot)}{\sin(\cdot)}\right)(y)\right|^2 \,dy dx \\
&\lesssim \int_{\mathbb{T}\setminus E} \left( (k_3 *|b^{(N)}_1|)(x) \right)^2  (k_3^2*|b^{(N)}_1|^2)(x)\,dx \\
&\lesssim \sup_{x \in \mathbb{T} \setminus E}\left( (k_3 *|b^{(N)}_1|)(x) \right)^2\int_{\mathbb{T}\setminus E}  (k_3^2*|b^{(N)}_1|^2)(x) \,dx \\
&\lesssim (\lambda \log N)^2 N\lambda \\
\end{align*}
Finally, we have
\begin{align*} &\frac{1}{N} \sum_{n=1}^N \int_{\mathbb{T}\setminus E} |S_nf(x)|^4\,dx \\
 &\lesssim \frac{1}{N} \sum_{n=1}^N \int_{\mathbb{T}\setminus E} |S_ng(x)|^4\,dx +\frac{1}{N} \sum_{n=1}^N \int_{\mathbb{T}\setminus E} |S_nb_1(x)|^4\,dx+\frac{1}{N} \sum_{n=1}^N \int_{\mathbb{T}\setminus E} |S_nb_2(x)|^4\,dx\\
&\lesssim \lambda^3+ \frac{1}{N} (\lambda \log N)^2 N\lambda + \lambda^3\\
&\lesssim \lambda^3 (\log N)^2
\end{align*} 
which is exactly what we want.
For $m \in \mathbb{N}$ and $m > 2$, we have the same Hilbert transform estimate for $g$ and $b_2$.  For $b_1$
\begin{align*}
\sum_{n=1}^N \int_{\mathbb{T}\setminus E} |S_nb_1(x)|^{2m} \,dx &=\sum_{n=1}^N \int_{\mathbb{T}\setminus E} |(S_nb_1(x))^m|^{2} \,dx\\ &\lesssim \sum_{n=1}^N \int_{\mathbb{T}\setminus E} \left|\prod_{i=1}^m \int_{\mathbb{T}} k_2(y_i)b_1(x-y_i)\,dy_i \right|^{2} \,dx \\
&\lesssim \int_{\mathbb{T}\setminus E} \int_{\mathbb{T}} |F_{x} * F_{x} * \cdots * F_{x}(y)|^{2} \,dy dx \\
& \lesssim \int_{\mathbb{T} \setminus E} \left(\int_{\mathbb{T}} |F_x(y_1)|\,dy_1\right)^2  \int_{\mathbb{T}} |F_{x} * F_{x} * \cdots * F_{x}(y_2)|^{2} \,dy_2 dx
\end{align*}
where $F_x(y)=\chi_{[-\frac{1}{N},\frac{1}{N}]^c}(x-y)\frac{b_1(y)}{\sin(x-y)}$, the second to last line contains $m-1$ convolutions, and the last line contains $m-2$ convolutions.  A small note to take into account is that the second inequality gives the bound a dependence on $m$ along with the following calculations.  Then
\begin{align*}
&\int_{\mathbb{T} \setminus E} \left(\int_{\mathbb{T}} |F_x(y_1)|\,dy_1\right)^2  \int_{\mathbb {T}} |F_{x} * F_{x} * \cdots * F_{x}(y_2)|^{2} \,dy_2 dx \\
&\lesssim \sup_{x  \in \mathbb{T} \setminus E} \left(\int_{\mathbb{T}} k_3(x-y)|b^{(N)}_1(y)|\,dy_1\right)^2 \int_{\mathbb{T} \setminus E} \int_{\mathbb{T}} |F_{x} * F_{x} * \cdots * F_{x}(y_2)|^{2} \,dy_2 dx \\
&\lesssim \lambda^2(\log N)^2 \int_{\mathbb{T} \setminus E} \int_{\mathbb{T}} |F_{x} * F_{x} * \cdots * F_{x}(y_2)|^{2} \,dy_2 dx
\end{align*}
and the rest of the even powers follow from repeating this procedure.  Interpolation gives the complete result.
\end{proof}
{\bf Remark:}
At the moment, the proof of Proposition \ref{p>2} is the best argument that is possible using the techniques of Theorem \ref{first thing}, but Proposition \ref{p>2} may not be the best formulation possible.
\section{Conclusion}
In conclusion, let us first revisit Theorem \ref{strong sums} by applying Theorem \ref{first thing} to the statement on convergence in measure of strong arithmetic means of Fourier series. 
\begin{proof}[Proof of Theorem \ref{strong sums} for $r=2$:] 
Define the operator
\begin{align*}
A_{N,2}f(x) := \sqrt{ \frac{1}{N}\sum_{n=1}^N |S_nf(x)|^2}.
\end{align*}
then by Theorem \ref{first thing}, $\sup_{N\geq 1} |\{ A_{N,2}f(x)>\lambda\}| \leq \frac{C\|f\|_1}{\lambda}$.  Fix $f \in L^1(\mathbb{T})$. We note that for any trigonometric polynomial $g$, it is obvious that  $\lim_{N \rightarrow \infty} \frac{1}{N}\sum_{n=1}^N |S_ng(x)-g(x)|^2 =0$ uniformly in $\mathbb{T}$.  Now let $\epsilon>0$ and let $g$ be a trigonometric polynomial such that $\|g-f\|_1 <\epsilon$.  Then $\frac{1}{M}\sum_{n=1}^M |S_ng(x)-g(x)|^2 \rightarrow 0$ as $M \rightarrow \infty$ uniformly in $\mathbb{T}$.  Therefore, let $N$ be such that $\frac{1}{M}\sum_{n=1}^M |S_ng(x)-g(x)|^2 <\sqrt{\epsilon}$ for all $M\geq N$.  Then
\begin{align*}
&\sup_{M\geq N}\left|\left\{ x \in \mathbb{T}| \sqrt{ \frac{1}{M}\sum_{n=1}^M |S_nf(x)-f(x)|^2}>\sqrt{\epsilon}  \right\}\right| \\
&\lesssim \sup_{M\geq N}\left|\left\{ \sqrt{ \frac{1}{M}\sum_{n=1}^M |S_nf(x)-S_ng(x)|^2}>\sqrt{\epsilon}  \right\}\right| \\
&\hspace{1cm}+\left|\left\{ \sup_{M\geq N}\sqrt{\frac{1}{M} \sum_{n=1}^M |g(x)-f(x)|^2}>\sqrt{\epsilon}  \right\}\right|\\
&\hspace{2cm}+\left|\left\{ \sup_{M\geq N}\sqrt{\frac{1}{M} \sum_{n=1}^M |S_ng(x)-g(x)|^2}>\sqrt{\epsilon}  \right\}\right|
\end{align*}
The second summand loses its dependence on $N, M$, and since $\frac{1}{M}\sum_{n=1}^M |S_ng(x)-g(x)|^2 <\sqrt{\epsilon}$, the last summand vanishes.  Then by removing a set $E$ from the first summand and using Markov's inequality we have
\begin{align*}
& \sup_{M}\left|\left\{ \sqrt{ \frac{1}{M}\sum_{n=1}^M |S_n(f-g)(x)|^2}>\sqrt{\epsilon}  \right\}\right| +\left|\left\{  |g(x)-f(x)|>\sqrt{\epsilon}  \right\}\right| \\
&\lesssim |E| +\frac{1}{\epsilon} \left(\sup_M \frac{1}{M} \sum_{n=1}^M \int_{\mathbb{T} \setminus E} |S_n(f-g)(x)|^2 \,dx \right)+ \frac{\|f-g\|_1}{\sqrt{\epsilon}}\\
&\lesssim  \sqrt{\epsilon}  +\frac{1}{\epsilon} \left(\frac{1}{\sqrt{\epsilon}}\|f-g\|_1^2 \right)+ \frac{\|f-g\|_1}{\sqrt{\epsilon}} \\
&\lesssim \sqrt{\epsilon}
\end{align*}
where the final inequalities come from letting $|E| \lesssim \sqrt{\epsilon}$. Thus,
\begin{align*}
\sup_{M\geq N}\left|\left\{ x \in \mathbb{T} \left|  \sqrt{ \frac{1}{M}\sum_{n=1}^M |S_nf(x)-f(x)|^2}>\sqrt{\epsilon} \right. \right\}\right| \lesssim \sqrt{\epsilon}
\end{align*} 
and this implies to the convergence in measure.  Thus, applying the fact that convergence in measure implies convergence almost everywhere up to a choice of subsequence ends the proof.
\end{proof}
This argument can be applied to Theorem \ref{strong sums high d} by utilizing Proposition \ref{high d}.  Therefore, let us employ Theorem \ref{strong sums high d} to prove Corollary \ref{converge high d}. The following is basically an extension of the $d=1$ case shown in (\cite{zygmund59}, Ch. 13):
\begin{proof}[Proof of Corollary \ref{converge high d}:]
Fix $f \in L^1(\mathbb{T}^d)$ by Theorem \ref{strong sums high d} there exists a subsequence $\{N_k\}$ so that for almost every $\theta \in \mathbb{T}^d$
\begin{align}\label{converge}
\frac{1}{N_k^d} \sum_{\overline{n} \in R^+_{N_k}} \left|S_{\overline{n}} f(\theta)-f(\theta)\right|^2 \xrightarrow{k \rightarrow \infty} 0.
\end{align}
Fix a $\theta \in \mathbb{T}^d$ for which convergence is satisfied and let $s_{\overline{n}} := S_{\overline{n}}f(\theta)$ and $s:=f(\theta)$.  Fix $\epsilon>0$. By \eqref{converge}, there exists $k_0$ such that for $k\geq k_0$
\begin{align*}
\frac{1}{N_k^d} \sum_{\overline{n} \in R^+_{N_k}} \left|s_{\overline{n}}-s\right|^2 <\epsilon^3.
\end{align*}
For any $N \in \mathbb{N}$, let $\nu(N)$ be the number of lattice points $\overline{n} \in R^+_N$ where $|s_{\overline{n}}-s|\geq \epsilon$ and let $\mu(N)$ be the number of lattice points $\overline{n} \in R^+_N$ where $|s_{\overline{n}}-s| <\epsilon$.  Then for $k\geq k_0$
\begin{align*}
\epsilon^3>\frac{1}{N_k^d} \sum_{\overline{n} \in R^+_{N_k}} \left|s_{\overline{n}}-s\right|^2 \geq \frac{\epsilon^2\nu(N_k)}{N_k^d}.
\end{align*}
Therefore, for $k \geq k_0$
\begin{align}\label{density}
\epsilon> \frac{\nu(N_k)}{N_k^d} \hspace{.5cm} \mbox{and} \hspace{.5cm} \frac{\mu(N_k)}{N_k^d} \geq 1-\epsilon.
\end{align}
For $m=1, 2, ...$, let $k_m\in  \mathbb{N}$ be defined so that $k_{m-1}<k_m$ and for $k \geq k_m$,  $\frac{1}{N_k^d} \sum_{\overline{n} \in R^+_{N_k}} \left|s_{\overline{n}}-s\right|^2 <\frac{1}{m^3}$.  Let $S_m$ be the set of lattice points such that $|s_{\overline{n}}-s|<\frac{1}{m}$ then
\begin{align}\label{inclusion}
S_1 \supset S_2 \supset \cdots \supset S_m \supset \cdots.
\end{align}
For each $m$, we will order the set $S_m \cap (R^+_{N_{k_{m+1}}}\setminus R^+_{N_{k_m}})$ so that 
\begin{align*}
\bigcup_{m \geq 1} S_m \cap (R^+_{N_{k_{m+1}}}\setminus R^+_{N_{k_{m}}})=:S=(s_{\overline{n}_k})
\end{align*}
is a sequence.  It is obvious that 
\begin{align*}
\lim_{k \rightarrow \infty} |s_{\overline{n}_k}-s| =0.
\end{align*}
It remains to show that
\begin{align*}
\lim_{k \rightarrow \infty} \frac{\#(S \cap [0, N_k]^d)}{N_k^d} =1.
\end{align*}
For a fixed $k \in \mathbb{N}$, let $\ell$ be such that $k_{\ell}< k\leq k_{\ell+1}$.  By the definition of $S$ and \eqref{inclusion}, $S \cap [0, N_k]^d \supset S_{\ell} \cap [0, N_k]^d$ and
\begin{align*}
\frac{\#(S \cap [0, N_k]^d)}{N_k^d} \geq \frac{\#(S_{\ell} \cap [0, N_k]^d)}{N_k^d}.
\end{align*}
By \eqref{density}
\begin{align*}
\frac{\#(S_{\ell} \cap [0, N_k]^d)}{N_k^d} \geq 1-\frac{1}{\ell}.
\end{align*}
Of course, $\frac{\#(S \cap [0, N_k]^d)}{N_k^d} \leq \frac{(N_k+1)^d}{N_k^d}$ and $\ell \rightarrow \infty$ as $N_k \rightarrow \infty$ so
\begin{align*}
\lim_{k \rightarrow \infty} \frac{\#(S \cap [0, N_k]^d)}{N_k^d} =1
\end{align*}
and we are done.
\end{proof}

\bibliographystyle{amsplain}

\begin{thebibliography}{10}

\bibitem{carl66}
{L.~Carleson}, {\em On convergence and growth of partial sums of Fourier series}, {Acta Math.} {\em 116} {(1966)}, {133-157}.

\bibitem{gog88}
{L.~D.~Gogoladze}, {\em On the strong summability almost everywhere}, {Sb. Math.} {\bf 135} {(1988)}, {158-169}.

\bibitem{hard13}
{G.~H.~Hardy}, {\em On the summability of Fourier's series}, {Proc. London Math.} {\bf 12} {(1913)}, {365-372}.

\bibitem{hun68}
 {R.~A.~Hunt}, {\em On the convergence of Fourier series}, {Orthogonal expansions and their continuous analogues},  {Southern Ill. University Press, Carbondale, Ill.},  {1968}, {pp. 235-255}.

\bibitem{kar06}
{G.~A.~Karagulyan}, {\em Everywhere divergent $\Phi$-means of Fourier series}, {Mathematical Notes} {\bf 8} {2006}, {no. 1}, {47-56}.

\bibitem{kolm23}
{A.~Kolmogoroff}, {\em Une s\'{e}rie de Fourier-Lebesgue divergente presque partout}, {Fund. Math.} {\bf 4} {(1923)},  {324-328}.

\bibitem{kon04}
{S.~V.~Konyagin}, {\em Convergent subsequences of partial sums of Fourier series of $\Phi(L)$}, {Orlicz centenary volume}, {Banach Center Publ. 64, Institute of Mathematics, Polish Academy of Sciences, Warsaw}, {2004}, {pp. 117-126}.

\bibitem{kon06}
{S.~V.~Konyagin}, {\em Almost everywhere convergence and divergence of Fourier series}, {Proceedings of the International Congress of Mathematicians, Madrid}, {2006}, {pp. 1393-1402}.

\bibitem{schlag13}
{C.~Muscalu and W.~Schlag}, {\em Classical and multilinear harmonic analysis}, {Cambridge Studies in Advanced Mathematics, No.137}, {vol. 1}, {2013}.

\bibitem{2schlag13}
{C.~Muscalu and W.~Schlag}, {\em Classical and multilinear harmonic analysis}, {Cambridge Studies in Advanced Mathematics, No.138}, {vol. 2}, {2013}.

\bibitem{rod90}
{V.~A.~Rodin}, {\em The space BMO and strong means of the Fourier series}, {Ann. Math.} {\bf 16} {(1990)}, {291-302}.

\bibitem{zygmund59}
{A.~Zygmund}, {\em Trigonometric series}, {second ed}, {vol. 1,2}, {Cambridge University Press}, {1959}.


\end{thebibliography}

\end{document}